\documentclass[11pt]{amsart}
\usepackage{amsmath,amssymb,amsfonts}
\usepackage[mathscr]{eucal}

\usepackage{enumerate}

\usepackage{tikz}
\usepackage{tikz-3dplot}
\usepackage{tikz-cd}
\usetikzlibrary{decorations.markings,decorations.pathmorphing,shapes}
\usepackage{xcolor}
\usepackage{lipsum}

\usepackage[textwidth=3cm, textsize=small, colorinlistoftodos]
{todonotes}

\theoremstyle{plain}
\newtheorem{theorem}{Theorem}[section]
\newtheorem{cor}[theorem]{Corollary}
\newtheorem{prop}[theorem]{Proposition}
\newtheorem{lemma}[theorem]{Lemma}

\theoremstyle{definition}
\newtheorem{definition}[theorem]{Definition}
\newtheorem{example}[theorem]{Example}
\newtheorem{remark}[theorem]{Remark}

\newtheorem{construction}[theorem]{Construction}

\newtheorem{notation}[theorem]{Notation}

\numberwithin{equation}{section}

\input xy
\xyoption{all}

\def\on{\operatorname}

\newcommand{\Aut}{\operatorname{Aut}}
\newcommand{\Band}{\operatorname{Band}}

\DeclareMathOperator*{\colim}{colim}

\newcommand{\Cyc}{\operatorname{Cyc}}
\newcommand{\Deltaop}{\Delta^{\op}}
\newcommand{\Equiv}{\operatorname{Equiv}}
\newcommand{\Ext}{\operatorname{Ext}}
\newcommand{\Ho}{\operatorname{Ho}}
\newcommand{\Hom}{\operatorname{Hom}}
\newcommand{\Homeo}{\operatorname{Homeo}}
\newcommand{\id}{\operatorname{id}}

\newcommand{\inj}{\operatorname{inj}}
\newcommand{\Iso}{\operatorname{Iso}}
\newcommand{\Lambdaop}{\Lambda^{\op}}
\newcommand{\Map}{\operatorname{Map}}
\newcommand{\Mon}{\operatorname{Mon}}

\newcommand{\ob}{\operatorname{ob}}
\newcommand{\op}{\operatorname{op}}
\newcommand{\proj}{\operatorname{proj}}
\newcommand{\Reedy}{\operatorname{Reedy}}
\newcommand{\Seg}{\operatorname{Seg}}
\newcommand{\Set}{\mathcal{S}\!\operatorname{et}}
\newcommand{\Sets}{\mathcal Set}
\newcommand{\SSets}{s\mathcal{S}\!\on{et}}
\newcommand{\Top}{\mathcal Top}
\def\ZZ{\mathbb{Z}}

\begin{document}

\title{Cyclic Segal spaces}

\author[J.E.\ Bergner]{Julia E.\ Bergner}
\author[W.H.\ Stern]{Walker H.\ Stern}

\address{Department of Mathematics, University of Virginia, Charlottesville, VA 22904}

\email{jeb2md@virginia.edu}
\email{walker@walkerstern.com}

\date{\today}

\renewcommand{\subjclassname}{\textup{2020} Mathematics Subject Classification}

\thanks{The first-named author was partially supported by NSF grant DMS-1906281. The second-named author was supported by the NSF Research Training Group at the University of Virginia under grant DMS-1839968.}

\begin{abstract}
In this survey article, we review some conceptual approaches to the cyclic category $\Lambda$, as well as its description as a crossed simplicial group.  We then give a new proof of the model structure on cyclic sets, work through the details of the generalized Reedy structure on cyclic spaces, and introduce model structures for cyclic Segal spaces and cyclic 2-Segal spaces.
\end{abstract}

\maketitle

\section{Introduction}

The cyclic category $\Lambda$, originally defined by Connes \cite{connes}, has been a useful framework for describing many different mathematical structures.  It can be described as a variant of the simplex category $\Delta$ with extra morphisms, or can be described in the framework of crossed simplicial groups, which was independently introduced in \cite{fl} and \cite{kras}.

More recently, cyclic structures have been studied within the context of 2-Segal spaces, as defined by Dyckerhoff and Kapranov \cite{dk}. While a (1-)Segal space roughly describes the structure of an up-to-homotopy topological category, in which both the set of objects and the set or morphisms is equipped with the structure of a topological space, a 2-Segal space encodes a weaker structure, in which composition need not be defined, or be unique (even up to homotopy) when it does exist.  Nonetheless, this composition is still associative, in an appropriate sense, and the structure of a 2-Segal space can be described via explicit combinatorial data.  

Since Segal spaces and 2-Segal spaces are given by simplicial spaces, or functors from $\Deltaop$ to the category of spaces, we can likewise consider cyclic Segal and 2-Segal spaces as functors out of $\Lambdaop$ instead.  Cyclic 2-Segal spaces are of significant interest; for instance, they provide space-valued invariants of oriented marked surfaces, as constructed in \cite[\S V.2]{dkcsg}, as well as Calabi-Yau algebras in span categories, as shown in \cite{stern}. 

This work is preliminary to providing an answer to the following question: which cyclic 2-Segal spaces arise as the output of an $S_\bullet$-construction?  For 2-Segal spaces, there is a Quillen equivalence of model categories between the model structure for 2-Segal spaces and a model structure for augmented stable double Segal spaces \cite{boors}.  A natural follow-up is then to identify which augmented stable double Segal spaces correspond to cyclic 2-Segal spaces.  We anticipate that they are described by an up-to-homotopy version of the symmetric double categories that we have shown correspond to cyclic 2-Segal sets in \cite{BSupcoming}.  Our next goal is to establish a model structure for them, and then show that the same functors used for 2-Segal spaces, and in particular the one given by a generalization of Waldhausen's $S_\bullet$-construction, give a Quillen equivalence between them.

One of the important features of cyclic spaces is that they correspond to spaces equipped with a circle action.  Thus, one motivation, aside from understanding the correspondence just described, is to give new insight into $S^1$-equivariant $K$-theory, given the central role that the $S_\bullet$-construction plays in algebraic $K$-theory.

The goal of this mostly expository paper is twofold.  First, we give an introduction to some of the different approaches to defining and understanding the category $\Lambda$, often in parallel with the analogous perspective applied to the simplicial category $\Delta$.  Second, we describe model structures on the categories of cyclic sets and cyclic spaces, culminating in appropriate model structures for cyclic Segal spaces and cyclic 2-Segal spaces.  

In the remainder of this introduction, we give a brief introduction to these ideas.

\subsection{The idea of the cyclic category} 

Our aim in this survey is to give an introductory treatment of cyclic structures, culminating in a theory of cyclic Segal spaces.  To reach the first part of this goal, we want to give multiple approaches to understanding Connes' cyclic category $\Lambda$, which was originally introduced in \cite{connes} to study cyclic cohomology, and the first few sections of this paper are devoted to a thorough understanding of this category.  Although Connes' original approach to the cyclic category was topological, which is the same perspective taken here, in practice the cyclic category often appears in the literature in terms of a more compact description via generators and relations. 

While we do, eventually, arrive at this generators-and-relations presentation, we do so after a leisurely tour through topology, first defining $\Lambda$ in terms of oriented circles and homotopy classes of maps between them, and then passing to a more combinatorial description. In this portion of the paper, we draw upon the approaches of Connes \cite{connes} and Dyckerhoff and Kapranov \cite[\S II]{dkcsg}, and in particular we aim to make more transparent the topological intuitions that underlie the category $\Lambda$.  

To arrive at the topological description of the cyclic category, we it is helpful first to understand how to reformulate the simplex category $\Delta$ in an analogous topological fashion. Fundamentally, the category $\Delta$ is the category of total orders, and it is this notion that we would like to encode topologically. To do so, we can equip a finite set $S$ with a linear order by embedding it in the oriented unit interval, as in Figure \ref{fig:points_on_interval}. 
\begin{figure}[h!]
	\begin{tikzpicture}
		\draw[(-)] (0,0) to (4,0);
		\foreach \x/\lab in {0.5/0,1/1,1.5/2,2/3,2.5/4,3/5,3.5/6}{
		\path[fill=red] (\x,0) circle (0.05);
		\path[red] (\x,0) node[label=above:$s_\lab$] {};
			};
	\end{tikzpicture}
	\caption{A linearly ordered set embedded in the open interval.}\label{fig:points_on_interval}
\end{figure}
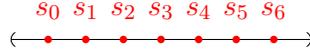  
If we are given two ordered sets $S$ and $T$, considered as subsets of the oriented unit interval, we can consider a self-homeomorphism of the unit interval $f \colon I\to I$ that sends $S$ into $T$. If we require that $f$ is also orientation-preserving, or equivalently, strictly increasing, then $f$ induces an injective monotone map $S\to T$.  We can thus write an injective morphism in $\Delta$ in terms of homeomorphisms of intervals. 

Unfortunately, we cannot obtain the non-injective maps of $\Delta$ in this framework.  We can fix this problem in two ways:  we can either relax the requirement that $f$ be increasing and instead ask that $f$ be monotone, i.e., non-decreasing; or we can replace the points above with intervals in $I$.  While we eventually discuss both of these potential solutions, we focus on the second here.   We can think of a totally ordered set as a collection of sub-intervals with disjoint closures in the oriented unit interval $I$, as pictured in Figure \ref{fig:sub_intervals}. 
\begin{figure}[h!]
	\begin{tikzpicture}
		\draw[(-)] (0,0) to (4,0);
		\foreach \x/\lab in {1/0,2/1,3/2}{
			\draw[red,(-)] ($(\x,0)-(0.4,0)$) to ($(\x,0)+(0.4,0)$);
			\path[red] (\x,0) node[label=above:$s_\lab$] {};
		};
	\end{tikzpicture}
\caption{Subintervals embedded in the unit interval.} \label{fig:sub_intervals}
\end{figure}
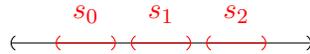
This additional data complicates the situation somewhat, since the ordered set we extract from such a collection of intervals $S\subseteq I$ is not simply the set $S$, but rather the set $\pi_0(S)$ of path-components of $S$. However, in exchange for this slight complication, we gain the ability to represent any monotone map in terms of homeomorphisms of the interval. 

If we have two such collections $S$ and $T$ of sub-intervals, any orientation-preserving homeomorphism $f \colon I\to I$ that sends $S$ into $T$ induces a monotone map $\pi_0(f) \colon \pi_0(S)\to \pi_0(T)$. Moreover, one can check that every monotone map $\pi_0(S)\to \pi_0(T)$ arises from such an orientation-preserving homeomorphism. 
\begin{figure}
	\begin{tikzpicture}
		\draw[(-)] (0,0) to (4,0);
		\foreach \x/\lab in {1/0,2/1,3/2}{
			\draw[red,(-)] ($(\x,0)-(0.4,0)$) to ($(\x,0)+(0.4,0)$);
			\path[red] (\x,0) node[label=above:$s_\lab$] {};
		};
		\draw[(-)] (0,-2) to (4,-2);
		\draw[(-),purple] (0.6,-2) to (0.9,-2);
		\path[purple] (0.75,-2) node[label=below:$t_0$] {};
		\draw[(-),purple] (1.1,-2) to (1.4,-2);
		\path[purple] (1.25,-2) node[label=below:$t_1$] {};
		\draw[(-),purple] (2.6,-2) to (3.4,-2);
		\path[purple] (3,-2) node[label=below:$t_2$] {};
		\draw[->] (2,-1.5) to node[label=left:$f$] {} (2,-0.5);
	\end{tikzpicture}
\caption{A depiction of a morphism in $\Delta_{\Top}$. For ease of visualization, the underlying homeomorphism of intervals is the identity. }\label{fig:map_in_Delta_as_subintervals}
\end{figure}
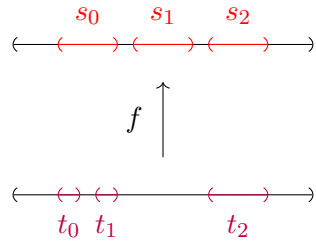
For example, the map depicted in Figure \ref{fig:map_in_Delta_as_subintervals} induces the monotone map that sends $t_0$ and $t_1$ to $s_0$, and sends $t_2$ to $s_2$.  Formalizing this construction allows us to define a topological category $\Delta_{\Top}$ whose objects are pairs $(I,S)$ and whose mapping spaces are spaces of self-homeomorphisms of intervals that preserve the chosen subsets. This topological category turns out to have contractible mapping spaces, and its homotopy category is equivalent to the simplex category $\Delta$. 

Now that we have a topological model for the simplex category, we can try to construct the cyclic category by first constructing the corresponding topological category and then taking its homotopy category. The natural question to ask, then, is ``with what should we replace the oriented interval?".  Since we want the cyclic category to encode cyclic group actions, which we can visualize as rotational symmetries, there is a natural choice: the circle $S^1$. 

Following our work with the simplex category, we can na\"ively extend to intervals on the circle.  We can define a topological category $\Lambda_{\Top}$, whose objects are pairs $(S^1,J)$ of spaces in which $J\subseteq S^1$ is a finite collection of open intervals with disjoint closures, and whose morphisms $f \colon (S^1,J)\to (S^1,K)$ are orientation-preserving homeomorphisms $f \colon S^1 \to S^1$ that map $J$ into $K$. 

Similarly to before, we can picture such morphisms, now viewing $f$ as mapping points radially outward, as in Figure \ref{fig:cyc_cat_morph}. 
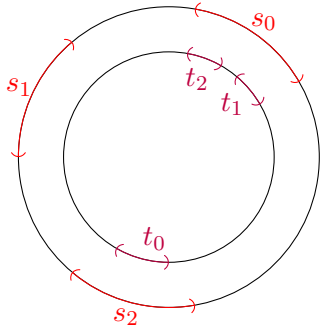
\begin{figure}[h!]
	\begin{tikzpicture}
		\draw (0,0) circle (2);
		\draw[red,(-)] (30:2) arc (30:80:2);
		\path[red] (55:2.2) node {$s_0$};
		\draw[red,(-)] (130:2) arc (130:180:2);
		\path[red] (155:2.2) node {$s_1$};
		\draw[red,(-)] (230:2) arc (230:280:2);
		\path[red] (255:2.2) node {$s_2$};
		\draw (0,0) circle (1.4);
		\draw[purple,(-)] (30:1.4) arc (30:50:1.4);
		\path[purple] (40:1.1) node {$t_1$};
		\draw[purple,(-)] (60:1.4) arc (60:80:1.4);
		\path[purple] (70:1.1) node {$t_2$};
		\draw[purple,(-)] (240:1.4) arc (240:270:1.4);
		\path[purple] (260:1.1) node {$t_0$};
	\end{tikzpicture}
	\caption{A morphism in the cyclic category, pictured in terms of the radial projection from the interior circle to the exterior circle.}\label{fig:cyc_cat_morph}
\end{figure}
The cyclic category $\Lambda$ is then defined to be the corresponding homotopy category. 

We summarize the correspondence between topological and combinatorial constructions in the following table. 
\begin{center}
	\begin{tabular}{c | p{0.6\linewidth}}
		\textbf{Combinatorial} & \textbf{Topological} \\ \hline 
		Linear order & Oriented line\\
		Cyclic order & Oriented circle\\
		Ordered set & Collection of subintervals\\
		Order-preserving map & Homotopy class of orientation-preserving homeomorphisms
	\end{tabular}
\end{center}

The more general framework of $G$-structured circles as described by Dyckerhoff and Kapranov in \cite[\S II]{dkcsg} can be used to apply the same basic intuition to access a more general class of orders, corresponding to planar Lie groups.  The perspective described above is sufficient for our purposes here.

\subsection{Cyclic objects and model structures}

Let us now turn to simplicial sets, given by functors $\Deltaop \rightarrow \Sets$, and cyclic sets $\Lambda \rightarrow \Sets$. Although every cyclic set has an underlying simplicial set, induced by the inclusion functor $\Deltaop \rightarrow \Lambdaop$, the two are quite different from one another, as can be highlighted by taking geometric realization.  

In particular, the geometric realization of a cyclic set is generally much larger than the geometric realization of its underlying simplicial set.  For example, the geometric realization of $\Delta[0]$, the simplicial set taking every object $[n]$ of $\Delta$ to a singleton set, is a single point.  On the other hand, the geometric realization of $\Lambda[0]$, again a constant diagram to a one-point set, is homotopy equivalent to a circle.  More generally, the geometric realization of the simplicial $n$-simplex $\Delta[n]$ is the topological $n$-simplex $\Delta^n$, while the geometric realization of the cyclic $n$-simplex $\Lambda[n]$ is $\Delta^n \times S^1$.  In particular, the former is a contractible space, while the latter is not.  Such subtleties carry over to the study of cyclic spaces.

In this paper, we give a new proof of a theorem of Dwyer, Hopkins, Kan \cite{dhk} that there is a model structure on the category of cyclic sets, and we recall its relationship to the category of topological spaces with an $S^1$-action.  We then give several model structures on the category of cyclic spaces.  First, we consider model structures in which the weak equivalences are levelwise weak equivalences of simplicial sets, given by standard approaches to model structures on categories of diagrams.  In the projective model structure, fibrations are taken to be levelwise fibrations of simplicial sets, while in the injective model structure the cofibrations are given levelwise.  Because $\Lambdaop$ has the structure of a generalized Reedy structure, in the sense of Berger and Moerdijk \cite{bm}, the category of cyclic spaces also has a generalized Reedy structure.  Although the existence of this model structure is a consequence of their general theorems, here we give an explicit description for cyclic spaces and specify generating sets of cofibrations and acyclic cofibrations.

We then give model structures on the category of cyclic spaces in which the fibrant objects are cyclic Segal spaces, each given by a localization of one of the above model structures with respect to a cyclic analogue of the spine inclusions used by Rezk to describe his model structure for Segal spaces \cite{rezk}.  Finally, in a similar way we produce model structures whose fibrant objects are cyclic 2-Segal spaces, analogously to the 2-Segal space model structure of Dyckerhoff and Kapranov \cite{dk}.

\subsection{Connections with $A_\infty$-structures} 

Although it is not our primary perspective here, we comment briefly here on the relationship between some of the structures here and $A_\infty$-structures.  

Recall that one model of the topological $A_\infty$-operad, which governs associative structures on spaces, is given by the ``little intervals" operad.  The reader familiar with $A_\infty$-algebras may recognize a similar phenomenon occurring in our presentation of $\Delta$ in terms of marked intervals.  There is a close relationship here, although we do not discuss it further here.

From another angle, composition in a Segal space or 2-Segal space is not required to be strictly associative, but rather coherently associative in a manner very much like that given by an $A_\infty$-structure.  Modeling such a coherent associativity was a key goal in the original definition of Segal spaces, but can also be understood as a motivation for the definition of 2-Segal spaces, as one aim of both \cite{dk} and \cite{gkt} is to study higher categorifications of algebraic structures, such as Hall algebras in the former and incidence coalgebras in the latter.  

\subsection{Outline of the paper}

In Section \ref{finitesubsetmodel} we describe our topological version of the simplex category and show that its homotopy category recovers the usual simplex category $\Delta$, and then in Section \ref{finitesubsetLambda} we use similar methods to give a topological version of the cyclic category $\Lambda$.  We take an alternate approach to these constructions and discuss duality in Sections \ref{intervalmodelDelta} and \ref{intervalmodelLambda}.  In Section \ref{csg} we investigate the category $\Lambda$ from the perspective of crossed simplicial groups.  Turning then to model structures, we give a new proof of the model structure on cyclic sets in Section \ref{cyclicsetmc}.  In Section \ref{levelwisemcs} we consider several model structures on cyclic spaces with levelwise weak equivalences that we localize in Section \ref{Segalmcs} and \ref{2Segalmcs} to obtain model structures for cyclic Segal and 2-Segal spaces.

\section{A finite subset model for the simplex category}  \label{finitesubsetmodel}

Taking the topological perspective that we began to describe in the introduction, our first goal is to give a definition of the simplex category in which the objects are discrete subsets of open intervals.  Before we begin, let us fix some notation that we use throughout. 

\begin{definition}
    An \emph{interval space} is one of the spaces $I=[0,1], I^\circ =(0,1), [0,1)$, and $(0,1]$. 
\end{definition} 

We equip these interval spaces with the orientation induced by the canonical orientation on $\mathbb R$, and we similarly regard the unit circle $S^1$ with its usual orientation. For topological spaces $X$ and $Y$, we denote by $C^0(X,Y)$ the space of continuous maps $X \rightarrow Y$, equipped with the compact-open topology. 

For our desired topological description of the simplex category, we do not want to consider arbitrary finite subsets of $I^\circ$,  but instead designate a fixed subset $\underline{n}$ of cardinality $n+1$ for every $n \geq 0$. The benefit of this approach is that its homotopy category is then isomorphic to the usual simplex category.  


\begin{definition}
    For $n\geq 0$, the \emph{standard $(n+1)$-pointed interval} is the pair of topological spaces $(I^\circ,\underline{n})$, where $\underline{n} \subseteq I^\circ$ is the discrete subspace 
    \[ \underline{n}:= \left\{\frac{1}{n+2},\ldots, \frac{n+1}{n+2} \right\}. \] 
    We denote the point $\frac{k+1}{n+2} \in \underline{n}$ by $k\in I^\circ$. 
\end{definition}

With these pairs as objects, we need to specify the morphisms.  As a consequence of our choice to work with discrete subsets, we need to consider maps that are less restrictive than self-homeomorphisms of the interval; we use monotone maps of intervals, which are necessarily homotopy equivalences. 

\begin{definition}
    For interval spaces $X$ and $Y$, a continuous map $f \colon X \rightarrow Y$ is \emph{monotone} if $f(x)\leq f(z)$ in $Y$ whenever $x \leq z$ in $X$. We denote by 
    \[ \Equiv^+(X,Y) \subseteq C^0(X,X) \]
    the space of monotone homotopy equivalences.  We denote by 
    \[ \Equiv^+((I^\circ,\underline{m}),(I^\circ,\underline{n}))\subseteq \Equiv^+(I^\circ,I^\circ) \]
    the subspace consisting of maps of pairs. 
\end{definition}

\begin{remark}
    There are several ways to draw a morphism 
    \[ f\in \Equiv^+((I^\circ,\underline{m}),(I^\circ, \underline{n})) \] 
    schematically; we describe two of them here. On the one hand, we can draw the graph of the morphism in $I^\circ \times I^\circ$, and on the other, we can draw the mapping cylinder, together with lines indicating the image of the marked points in $I^\circ\times I$. For example, the two pictures in Figure \ref{fig:graph_and_map_cylinder_2_to_3} can both be used to represent the same piecewise linear topological map $f \colon (I^\circ,\underline{2})\to (I^\circ,\underline{3})$.
    \begin{figure}[h!]
		\begin{tabular}{ c  c  c}
			\textbf{Graph} & \phantom{middle} & \textbf{Mapping Cylinder} \\ 
			
			\begin{tikzpicture}
				\path[white] (0,4.2) circle (0.05);
				\draw (0,0) rectangle (4,4);
				\foreach \x/\lab in {1/0,2/1,3/2}{
				\path[fill=blue] (\x,0) circle (0.05);
				\path[blue] (\x,0) node[label=below:$\mathbf{\lab}$] {};
				};
			\foreach \y/\lab in {0.8/0,1.6/1,2.4/2,3.2/3}{
				\path[fill=blue] (0,\y) circle (0.05);
				\path[blue] (0,\y) node[label=left:$\mathbf{\lab}$] {};
			};
			\draw[red] (0,0) to (1,1.6) to (2,1.6) to (3,3.2) to (4,4);
			\draw[dashed, blue] (1,0) to (1,1.6) (2,0) to (2,1.6) (3,0) to (3,3.2); 
			\end{tikzpicture} &  &
			\begin{tikzpicture}
				\draw (0,0) rectangle (4,4); 
				\path[white] (2,0) node[label=below:$0$]{};
				\foreach \y/\lab in {0.8/0,1.6/1,2.4/2,3.2/3}{
					\path[fill=blue] (4,\y) circle (0.05);
					\path[blue] (4,\y) node[label=right:$\mathbf{\lab}$] {};
				};
			\foreach \x/\lab in {1/0,2/1,3/2}{
				\path[fill=blue] (0,\x) circle (0.05);
				\path[blue] (0,\x) node[label=left:$\mathbf{\lab}$] {};
			};
			\draw[red] (0,1) to (4,1.6) (0,2) to (4,1.6);
			\draw[red] (0,3) to (4,3.2); 
			\end{tikzpicture}
		\end{tabular}
		\caption{Graph an mapping cylinder depictions of the same map of marked intervals.} \label{fig:graph_and_map_cylinder_2_to_3}
	\end{figure}
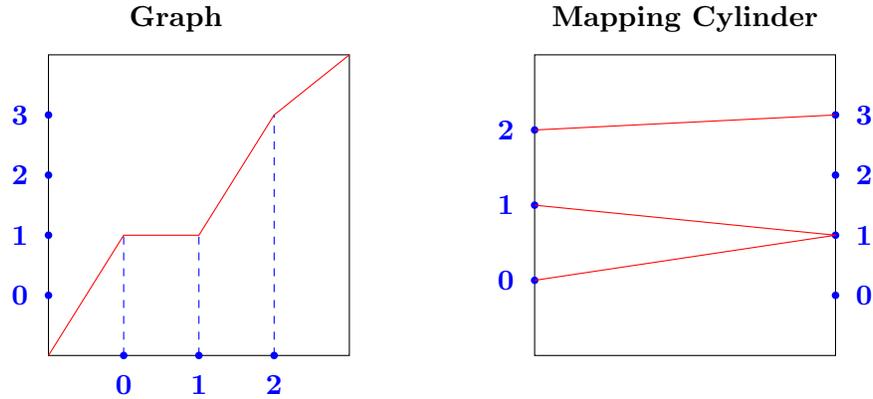
    While the former has the benefit of possibly being more familiar, the latter better resembles the pictures one might draw to represent maps in the simplex category. 
\end{remark}

We now are able to give our chosen definition of the simplex category.

\begin{definition} \label{topsimplexcat}
    The \emph{topological simplex category} is the topological category $\Delta_{\Top}$ with objects $(I^\circ,\underline{n})$ for each $n\geq 0$ and morphism spaces
    \[ \Map_{\Delta_{\Top}}((I^\circ, \underline{n}), (I^\circ, \underline{m}))= \Equiv^+((I^\circ, \underline{n}), (I^\circ,\underline{m})). \]
    The \emph{simplex category} $\Delta$ is the homotopy category of $\Delta_{\Top}$, defined by taking path components of mapping spaces.
\end{definition}

We now want to demonstrate that this definition is a sensible one for the simplex category, for which we recall its the usual combinatorial definition.


\begin{definition}
    For any $n \geq 0$, the \emph{standard ordinal} $[n]$ is the totally ordered set 
    \[ [n]= \{0<1<\cdots<n\}. \]
    The simplex category $\Delta$ the category whose objects are the standard ordinals and whose morphisms are the monotone maps between them.  We denote the set of monotone maps from $[n]$ to $[m]$ by $\Mon([n],[m])$.  
\end{definition}

We want to establish an isomorphism between the category given by this definition and the one given in Definition \ref{topsimplexcat}.  The first step is the following lemma.

\begin{lemma} \label{lem:equivcontract} 
    For any interval space $X$, the space $\Equiv^+(X,X)$ is contractible. 
\end{lemma}
 
 
\begin{proof}
    Straight-line homotopies provide an explicit contraction of the space $\Equiv^+(X,X)$.  
\end{proof}
 
The following proposition follows directly.

\begin{prop} \label{prop:equiv_comps_monotone}
    For any $n,m\geq 0$, there is a homotopy equivalence 
    \[ \Equiv^+((I^\circ, \underline{n}), (I^\circ,\underline{m})) \simeq \Mon([n],[m]), \]
    where $\Mon([n],[m])$ is viewed as a discrete space.
\end{prop}

Thus we have retrieved the more conventional definition of the simplex category from our topological definition. In what follows, we use the conventional notation in $\Delta$, for example denoting the object corresponding to $(I^\circ,\underline{n})$ by $[n]$ and freely using the identification of $\Hom_{\Delta}([n],[m])$ with $\Mon([n],[m])$ provided by Proposition \ref{prop:equiv_comps_monotone}. 

\begin{remark}\label{rmk:Simp_idents}
    A conventional way to present the simplex category is via generators and relations, a presentation that can be derived from the combinatorial description above. The simplex category $\Delta$ is generated by the \emph{face morphisms}
    \[ \begin{tikzcd}[row sep=0em,ampersand replacement=\&]
		d_i \colon \&[-3em] {[n-1]}\arrow[r] \& {[n]}\\
		\& j \arrow[r,mapsto] \& \begin{cases}
			j & j<i\\
			j+1 & j\geq i, 
		\end{cases}
	\end{tikzcd} \]
    that skip $i$ in the image and the \emph{degeneracy morphisms} 
    \[ \begin{tikzcd}[row sep=0em,ampersand replacement=\&]
		s_i \colon \&[-3em] {[n+1]}\arrow[r] \& {[n]}\\
		 \& j \arrow[r,mapsto] \& \begin{cases}
		 	j & j\leq i \\
		 	j-1 & j>i 
		 \end{cases}
    \end{tikzcd} \]
    that send two consecutive points to $i$. These morphisms are subject to the \emph{simplicial identities}:
	\begin{align*}
		d_j\circ d_i&=d_i\circ d_{j-1} & i<j\\
		s_j\circ s_i&=s_i\circ s_{j+1} & i\leq j\\
		s_j\circ d_i&=
		\begin{cases}
			d_i \circ s_{j-1} & i<j\\
			\id_{[n]} & i=j,j+1\\
			d_{i-1}\circ s_j & i>j+1.
		\end{cases} 
	\end{align*}
\end{remark}


We conclude this section with a discussion of the approach that we took in this section, as we repeat it frequently throughout this paper.  We defined a topological category $\Delta_{\Top}$ with objects given by the natural numbers and morphism spaces given by spaces of monotone homotopy equivalences.  The conventional description of $\Delta$, however, is as an ordinary category, with morphism sets given by the monotone maps between finite ordered sets.  We thus have a general means to compare an ordinary category, with discrete morphism spaces, to a topological category with possibly more interesting spaces of morphisms.  One way to rephrase Lemma \ref{lem:equivcontract} is to say that the mapping spaces in $\Delta_{\Top}$ are \emph{homotopy discrete}, in that they are homotopy equivalent to discrete spaces.

To tease further meaning from this statement, we describe some constructions for topological categories.

\begin{definition}
If $\mathcal{C}$ is a topological category, its \emph{homotopy category} is the category $\Ho(\mathcal{C})$ with the same objects as $\mathcal{C}$ and hom-sets given by 
\[ \Hom_{\Ho}(\mathcal{C})(x,y):= \pi_0 \Map_{\mathcal{C}}(x,y). \]
Since the functor $\pi_0$ preserves products, we can define the composition in $\Ho(\mathcal{C})$ to be the image under $\pi_0$ of the composition in $\mathcal{C}$. The identities are then the path components of the identity. 
\end{definition}

Since $\pi_0$ is functorial and preserves products, the construction of the homotopy category is functorial.  For our purposes, the important fact is that, given a functor of topological categories $F \colon \mathcal{C}\to \mathcal{D}$, we obtain an induced functor on homotopy categories
\[ \begin{tikzcd}
\Ho(F) \colon &[-3em] \Ho(\mathcal{C}) \arrow[r] & \Ho(\mathcal{D}). 
\end{tikzcd} \]

\begin{definition}
Let $\mathcal C$ and $\mathcal D$ be topological categories and $F \colon \mathcal C \rightarrow \mathcal D$ a topological functor, so that for any objects $X$ and $Y$ of $\mathcal C$, there is a continuous map
\[ \Map_\mathcal C(X,Y) \rightarrow \Map_\mathcal D(FX, FY). \]
Such a functor $F$ is a \emph{Dwyer-Kan equivalence} if:
\begin{enumerate}
    \item the above map is a weak homotopy equivalence of spaces for all objects $X$ and $Y$; and
    
    \item the induced functor $\Ho(F)$ on homotopy categories is essentially surjective, so that any object of $\Ho(\mathcal D)$ is isomorphic to one in the image of $\Ho(F)$. 
\end{enumerate}
\end{definition}

In the example of the simplex category, the second condition was satisfied immediately, since the two categories $\Delta_{\Top}$ and $\Delta$ have exactly the same set of objects.  The arguments in this section can be interpreted as establishing the following result.

\begin{prop} \label{DeltaTopDK}
    The map $\Delta_{\Top} \rightarrow \Delta$ is a Dwyer-Kan equivalence.
\end{prop}

It is important to note that $F$ being a Dwyer-Kan equivalence is strictly stronger than being an equivalence of homotopy categories, precisely because it captures information about the homotopy types of all of the mapping spaces.

\section{The cyclic category in terms of finite subsets} \label{finitesubsetLambda}

We now turn to the cyclic category $\Lambda$.  As in our exploration of the simplex category, we want to define specific subsets of $S^1$ as the objects in our topological version of the $\Lambda$. As a first step, we define particular covering maps $\mathbb{R}\to S^1$ where, here and throughout, we identify $S^1$ with the unit circle in $\mathbb{C}$. 

\begin{definition}
    For any $n\geq 0$, the \emph{$(n+1)$-stretched universal cover} of the circle is the map
    \[ \begin{tikzcd}[row sep=0em]
        p_{n+1} \colon &[-3em] \mathbb R \arrow[r] & S^1 \\
        &  x \arrow[r,mapsto] & \exp\left( \frac{2\pi ix}{n+1}\right).
	\end{tikzcd} \]
    We denote by $\underline{n}$ the image of $\mathbb Z \subseteq \mathbb R$ under $p_{n+1}$, For simplicity, for any $0\leq i \leq n$ we denote the image $p_{n+1}(i)$ by $i\in \underline{n} \subseteq S^1$ .
\end{definition}

As a consequence of this definition, we have the following convenient notion of monotone maps. 

\begin{definition}
    Denote by $[0,n+1] \subseteq \mathbb R^n$ the closed interval. A continuous map $f \colon S^1\to S^1$ is \emph{monotone} if the lift 
    \[ \begin{tikzcd}
		& & \mathbb R \arrow[d,"p_{k+1}"] & \\
		{[0,n+1]}\arrow[r,"p_{n+1}"']\arrow[urr,dashed, "\widetilde{f}"] & S^1\arrow[r,"f"'] &S^1 
	\end{tikzcd} \]  
    to any (and thus every) universal cover $p_{k+1}$ is a monotone map. We denote by 
    \[  \Equiv^+(S^1,S^1)\subseteq C^0(S^1,S^1) \]
    the subspace consisting of monotone homotopy equivalences. 
\end{definition}

On an intuitive level, the idea here is that a map $f\in \Equiv^+(S^1,S^1)$ is a loop that traces precisely once around the target circle, possibly stopping for some time, but never backtracking. In particular, any such $f$ has degree 1. As with the simplex category, we can restrict to maps of pairs to define the topological cyclic category, as follows.

For $n\geq 0$, denote by $\langle n\rangle$ the pair of spaces $(S^1,\underline{n})$. For $n,m>0$, we denote by
\[ \Equiv^+(\langle n\rangle,\langle m\rangle)=\Equiv^+((S^1,\underline{n}),(S^1,\underline{m}))\subseteq \Equiv^+(S^1,S^1) \]
the subspace consisting of maps of pairs, i.e., maps $f$ such that $f(\langle n\rangle)\subseteq \langle m\rangle$. 

\begin{definition}
    The \emph{topological cyclic category} $\Lambda_{\Top}$ is the topological category with objects $\langle n\rangle$ for $n\geq0$, and hom-spaces 
    \[ \Map_{\Lambda_{\Top}}(\langle n\rangle,\langle m\rangle)=\Equiv^+(\langle n\rangle,\langle m\rangle). \]
    Finally, for any $f \colon \langle n\rangle \to \langle m\rangle$ in $\Lambda_{\Top}$, we denote by $\underline{f} \colon \underline{n}\to \underline{m}$ the underlying map of sets. 
\end{definition}

Precisely as with the simplex category, we now take path components to obtain the cyclic category. 

\begin{definition}
    The \emph{cyclic category} $\Lambda$ is the homotopy category of $\Lambda_{\Top}$. That is, $\Lambda$ has objects $\langle n\rangle$ for $n\geq 0$ and morphisms
    \[ \Hom_{\Lambda}(\langle n\rangle,\langle m\rangle) =\pi_0 \Equiv^+(\langle n\rangle,\langle m\rangle). \]
\end{definition}

As with the simplex category, the following proposition demonstrates that passing to $\Lambda$ from $\Lambda_{\Top}$ does not lose homotopical information. 

\begin{prop} 
    The mapping spaces in $\Lambda_{\Top}$ are homotopy discrete. In particular, the canonical functor $\Lambda_{\Top}\to \Lambda$ is a Dwyer-Kan equivalence.  
\end{prop}

\begin{proof}
    The proof amounts to an examination of straight-line homotopies of induced maps on universal covers. 
\end{proof}

\subsection{Drawing morphisms}  

To represent a morphism in $\Lambda$ visually, it suffices to draw a single map of marked circles representing that morphism.  As in the case of the category $\Delta$, there are a number of pictorial ways to represent such a map, of which we focus on two: mapping cylinders and graphs of lifts to universal covers.

A monotone homotopy equivalence $f \colon (S^1,\underline{n})\to (S^1,\underline{m})$ that is also a morphism of pairs is uniquely determined by the composite
\[ \begin{tikzcd}
    \overline{f} \colon &[-3em] {[0,n+1]} \arrow[r,"p_{n+1}"] & S^1 \arrow[r,"f"] &  S^1.
\end{tikzcd} \]
Since we must have $\overline{f}(0)= i$ for some $0\leq i \leq m$, $\overline{f}$ uniquely determines and is uniquely determined by the lift 
\[ \begin{tikzcd}
    & \mathbb R \arrow[d,"p_{m+1}"]  \\
    {[0,n+1]}\arrow[ur,dashed, "\widetilde{f}"] \arrow[r,"\overline{f}"'] &S^1. 
\end{tikzcd} \]
Since $f$ must be degree 1, this lift can be alternately viewed as a map 
\[ \begin{tikzcd}
    \widetilde{f} \colon &[-3em] {[0,n+1]}\arrow[r] & {[i,m+1+i]} 
\end{tikzcd} \]
that sends the integer points $\{0,1,\ldots,n\}$ to integer points in $\{i,i+1,\ldots,i+m+1\}$. 








\begin{example} \label{eg:include}
    We define a morphism $f \colon \langle 1\rangle \rightarrow \langle 2 \rangle$ whose underlying map of sets is given by $0 \mapsto 2$ and $1\mapsto 1$. Let us describe why why this description uniquely determines the homotopy class $f$. 
	
    Since $f$ is monotone of degree 1, the lift 
    \[ \begin{tikzcd}
		& \mathbb{R}\arrow[d,"p_3"]\\
		{[0,2]} \arrow[r,"f"']\arrow[ur,"\widetilde{f}"] & S^1 
	\end{tikzcd} \]
    to the universal cover is non-decreasing. Recall that the  covering map $p_{n+1}$ associated with $(S^1,\underline{n})$ is the map $\mathbb{R}\to S^1$ given by $x\mapsto x\mod n+1$. Since $f$ is degree 1 and monotone, $f$ is uniquely determined by the map 
    \[ \begin{tikzcd}
		\widetilde{f} \colon &[-3em] {[0,2]} \arrow[r] & {[\widetilde{f}(0),\widetilde{f}(0)+3].}
	\end{tikzcd} \]
    Since $f(0)=2$ and $f(1)=1$, we see that $\widetilde{f}(1)=\widetilde{f}(0)+2$. Similarly, $\widetilde{f}(2) =\widetilde{f}(0)+3$. In each of these cases, there is a unique homotopy class of endpoint-preserving monotone maps of intervals
    \[ \begin{tikzcd}
		{[0,1]} \arrow[r] & {[\widetilde{f}(0), \widetilde{f}(0)+2]}
	\end{tikzcd} \]
    and 
    \[ \begin{tikzcd}
	{[1,2]} \arrow[r] & {[\widetilde{f}(0)+2, \widetilde{f}            (0)+3]}
	\end{tikzcd} \]
    and so the homotopy class of $f$ is uniquely determined. 
	
    Using the first method of sketching morphisms, we draw the graph of the map 
    \[ \begin{tikzcd}
	\widetilde{f} \colon &[-3em] {[0,2]}\arrow[r] & {[2,5]}
	\end{tikzcd} \]
	with marked points indicated on the two axes in Figure \ref{fig:graph_1_to_2}. 
	\begin{figure}[h!]
		\begin{tikzpicture}
			\draw (0,1) to (0,7);
			\path[red] (0,1) node[label=left:$(2)\quad  2$] {};
			\path[red] (0,3) node[label=left:$(0)\quad 3$] {};
			\path[red] (0,5) node[label=left:$(1)\quad 4$] {};
			\foreach \y in {1,3,5}{
				\draw[red,fill=red] (0,\y) circle (0.05);
			};
			\draw (2,-1) to (6,-1);
			\path[red] (2,-1) node[label=below:$0$] {};
			\path[red] (4,-1) node[label=below:$1$] {};
			\foreach \x in {2,4}{
				\draw[red,fill=red] (\x,-1) circle (0.05);
			}
			
			\draw[opacity=0.5] (2,1) rectangle (6,7);
			\draw[blue] (2,1) to (4,3);
			\draw[blue] (4,3) to (6,7);
			\draw[thin,red,dashed] (2,-1) to (2,1);
			\draw[thin,red,dashed] (4,-1) to (4,3);
		\end{tikzpicture}
		\caption{A graph depiction of the morphism $f \colon \langle 1\rangle \to \langle 2\rangle$ of Example \ref{eg:include}.} \label{fig:graph_1_to_2} 
	\end{figure}
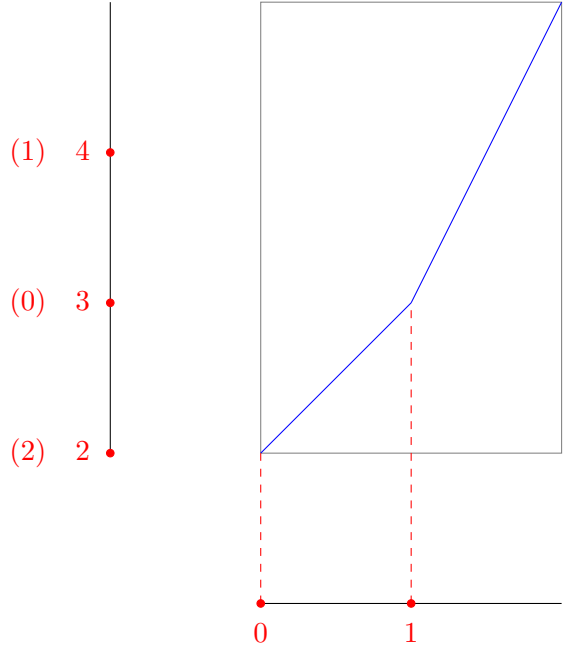
	
    Our second method requires us to draw the mapping cylinder, which is the pushout of the diagram 
    \[ \begin{tikzcd}
    S^1\times\{1\}\arrow[d,"f"'] \arrow[r,hookrightarrow] & S^1 \times I \arrow[d]\\
    S^1 \arrow[r] & M_f.
    \end{tikzcd} \] 
    This mapping cylinder is pictured in Figure \ref{fig:3d_mapping_cyl_1_to_2}.  
    \begin{figure}[h!]
    \includegraphics[width=0.7\textwidth]{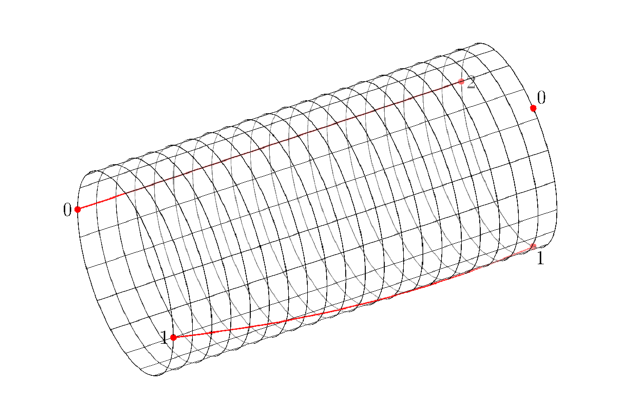}
    \caption{A mapping cylinder depiction of a morphism $f \colon \langle 1\rangle \to \langle 2\rangle$.} \label{fig:3d_mapping_cyl_1_to_2}
    \end{figure} 
\end{example}

\begin{example} \label{eg:collapseto1}
    Our second example is a morphism $g \colon \langle 2\rangle \rightarrow \langle 0\rangle$ whose underlying map of sets must send every marked point to $0$. We want to show that, in this case, the underlying map of sets does not uniquely determine the morphism $g$. 
	
    In Figure \ref{fig:mapping_cylinder_2_to_1} we draw the mapping cylinder depiction of the morphism we want to consider.
    \begin{figure}[h!]
	\includegraphics[width=0.7\textwidth]{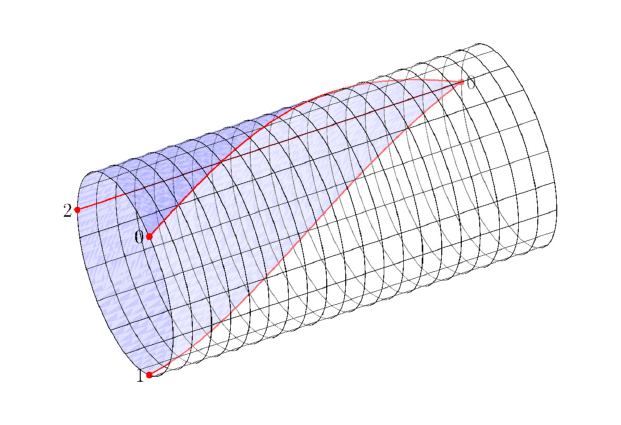}
		\caption{A mapping cylinder depiction of a morphism $\langle 2\rangle \to \langle 0\rangle$.} \label{fig:mapping_cylinder_2_to_1}
    \end{figure}
    We have highlighted in blue the line segment in $(S^1,\underline{2})$ that gets collapsed to the point $0$. Observe that we can recover $g$ if we know two pieces of information:
    \begin{enumerate}
		\item that $g$ sends every marked point to $0$; and
		
		\item that the oriented interval in $(S^1,\underline{2} )$ going from $1$ to $0$ via the point $2$ is collapsed by $g$. We could equivalently remember that the oriented interval in $(S^1,\underline{2})$ that goes from $0$ to $1$ is not collapsed.
    \end{enumerate}
	
    For completeness, we include a picture in graph form in Figure \ref{fig:graph_2_to_1}. 
    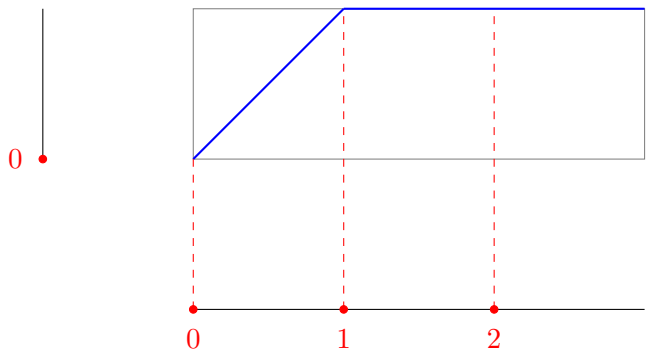
\begin{figure}[h!]
		\begin{tikzpicture}
			\draw (0,1) to (0,3);
			\path[red] (0,1) node[label=left:$0$] {};
			\foreach \y in {1}{
				\draw[red,fill=red] (0,\y) circle (0.05);
			};
			\draw (2,-1) to (8,-1);
			\path[red] (2,-1) node[label=below:$0$] {};
			\path[red] (4,-1) node[label=below:$1$] {};
			\path[red] (6,-1) node[label=below:$2$] {};
			
			\foreach \x in {2,4,6}{
				\draw[red,fill=red] (\x,-1) circle (0.05);
			}
			
			\draw[opacity=0.5] (2,1) rectangle (8,3);
			\draw[thick,blue] (2,1) to (4,3);
			\draw[thick,blue] (4,3) to (8,3);
			\draw[thin,red,dashed] (2,-1) to (2,1);
			\draw[thin,red,dashed] (4,-1) to (4,3);
			\draw[thin,red,dashed] (6,-1) to (6,3);
		\end{tikzpicture}
		\caption{A graph depiction of a morphism $\langle r\rangle \to \langle 1\rangle$.}\label{fig:graph_2_to_1}
    \end{figure}
\end{example}

\begin{example} \label{eg:Two_morphs_1to0}
    As one final example, let us try to classify the possible morphisms $\langle 1\rangle \to \langle 0\rangle$. There is only one possible underlying map of sets, and so the question is which possible lifts could induce this map. Such lifts are monotone continuous maps  
    \[ \begin{tikzcd}
    \widetilde{f}\colon &[-3em] {[0,2]}\arrow[r] & {[0,1]}
    \end{tikzcd} \] 
    such that $0\mapsto 0$, $2\mapsto 1$, $1$ is sent to either $0$ or $1$. It is not hard to see that there are only two such lifts, up to homotopy; we draw their graphs in blue and purple in Figure \ref{fig:graph_1_to_0}. 
    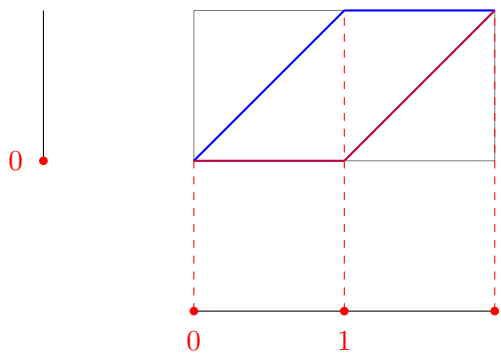
\begin{figure}[h!]
		\begin{tikzpicture}
			\draw (0,1) to (0,3);
			\path[red] (0,1) node[label=left:$0$] {};
			\foreach \y in {1}{
				\draw[red,fill=red] (0,\y) circle (0.05);
			};
			\draw (2,-1) to (6,-1);
			\path[red] (2,-1) node[label=below:$0$] {};
			\path[red] (4,-1) node[label=below:$1$] {};
			
			\foreach \x in {2,4,6}{
				\draw[red,fill=red] (\x,-1) circle (0.05);
			}
			
			\draw[opacity=0.5] (2,1) rectangle (6,3);
			\draw[thick,blue] (2,1) to (4,3);
			\draw[thick,blue] (4,3) to (6,3);
			\draw[thick,purple] (2,1) to (4,1);
			\draw[thick,purple] (4,1) to (6,3);
			\draw[thin,red,dashed] (2,-1) to (2,1);
			\draw[thin,red,dashed] (4,-1) to (4,3);
			\draw[thin,red,dashed] (6,-1) to (6,3);
		\end{tikzpicture}
        \caption{A graph depiction of lifts of the map $\langle 1 \rangle \rightarrow \langle 0 \rangle$.} \label{fig:graph_1_to_0}
	\end{figure}
	Notice that, in the corresponding map of circles $f \colon (S^1,\underline{1})\to (S^1,\underline{0})$, choosing which of these two morphisms we are representing amounts to choosing which segment of $(S^1,\underline{1})$ is not collapsed to a point by $f$. 
\end{example}

\subsection{Alternative approaches}  

We conclude with two further ways to think about the morphisms in $\Lambda$, the first of which gives a particularly nice combinatorial description of the morphisms, and the second of which can be described in terms of underlying sets.  

There is, in fact, another kind of map lurking in the background of our previous discussions, one that provides one of the nicest combinatorial characterizations of $\Lambda$. Deck transformations along $p_{n+1}$ and $p_{m+1}$ define two $\ZZ$-actions on $\ZZ$, where $1\in\ZZ$ acts by addition by $n+1$ and $m+1$, respectively. We denote the former $\ZZ$-set by ${}^{n\circlearrowright}\ZZ$ and the latter by ${}^{m\circlearrowright}\ZZ$. Lifting $f \colon \langle n\rangle \to \langle m\rangle$ to a morphism $\widetilde{f}$ of universal covers then yields a $\ZZ$-equivariant map 
\[ \begin{tikzcd}
    \widetilde{f}\colon &[-3em] {}^{n\circlearrowright}\ZZ \arrow[r] & {}^{m\circlearrowright}\ZZ.
\end{tikzcd} \]

This map completely determines the path component of $f$. However, since multiple lifts determine the same map, multiple $\ZZ$-equivariant maps determine the same path component. To obtain a first combinatorial characterization of the morphisms in $\Lambda$, we make the following definition. 

\begin{definition}
    We denote the set of monotone $\ZZ$-equivariant maps from ${}^{n\circlearrowright}\ZZ$ to ${}^{m\circlearrowright}\ZZ$ by $\ZZ\!\Mon({}^{n\circlearrowright}\ZZ,{}^{m\circlearrowright}\ZZ)$. We define a $\ZZ$-action on $\ZZ\!\Mon({}^{n\circlearrowright}\ZZ,{}^{m\circlearrowright}\ZZ)$ by 
    \[ (1+f)(k)=f(k+n+1)=f(k)+m+1. \]
\end{definition}

An examination of which maps of universal covers induce the same map then yields the following result. 

\begin{prop}
    Let $n,m\geq 0$. There is a canonical identification 
    \[ \Hom_{\Lambda}(\langle n\rangle,\langle m\rangle) \cong \ZZ\!\Mon({}^{n\circlearrowright}\ZZ,{}^{m\circlearrowright}\ZZ)/{\ZZ}. \]
    Moreover, this identification is compatible with composition. 
\end{prop}

\begin{remark}
    We can define a variant of the cyclic category whose objects are the same, but whose hom-sets are $\ZZ\!\Mon({}^{n\circlearrowright}\ZZ,{}^{m\circlearrowright}\ZZ)$. This category is known in the literature as the \emph{paracyclic category}, and it often denoted by $\Lambda_\infty$ or $\Delta\ZZ $. See \cite[Proposition 6.3.4(c)]{loday}, \cite[Example I.23]{dkcsg}, or \cite[Appendix B]{scholzenikolaus} for more details.  
\end{remark}


Now let us consider the other combinatorial way to encode morphisms in $\Lambda$, which is in terms of morphisms of underlying sets. There is a forgetful functor 
\[ \begin{tikzcd}
    U \colon &[-3em] \Lambda\arrow[r] & \Set
\end{tikzcd} \]
that sends $\langle n\rangle=(S^1,\underline{n})$ to the set $\underline{n}$, canonically identified with $\{0,1\ldots,n\}$, and sends a morphism $f \colon \langle n\rangle \to\langle m\rangle$ to the induced map $\underline{f} \colon \underline{n}\to \underline{m}$ on marked points. 

We begin by ascertaining which maps of sets are the underlying maps of morphisms in $\Lambda$. 

\begin{lemma} \label{lem:underlying_cyc_maps}
    A map $\varphi \colon \underline{n}\to\underline{m}$ is the underlying map of a morphism in $\Lambda$ if and only if there exist $i\in \underline{n}$ and $j\in\underline{m}$ such that $\varphi$ is weakly monotone when viewed as a map 
    \[ \begin{tikzcd}
    \{i+1,\ldots,n,0,\ldots,i\}\arrow[r] & \{j+1,\ldots,n,0,\ldots, j\}.
    \end{tikzcd} \]
\end{lemma}

\begin{proof}
    The condition is necessary, since we can cut out circles at an unmarked point to obtain a map of marked intervals. 
	
    On the other hand, given such $\varphi$, $i$, and $j$, we can represent $\varphi$ by a monotone map of marked intervals 
    \[ \begin{tikzcd}
    f \colon &[-3em] (I^\circ, \{i+1,\ldots,i\}) \arrow[r] & (I^\circ, \{j+1,\ldots,j\}).
    \end{tikzcd} \]
    The induced map between one-point compactifications can be identified with a map $\langle n\rangle \to\langle m\rangle$, and the underlying map is still $\varphi$. 
\end{proof}

\begin{notation}
    We denote the set of maps satisfying the condition of Lemma \ref{lem:underlying_cyc_maps} by $\Cyc(\underline{n}, \underline{m})$.
\end{notation}

The final step in our second combinatorial description is the following. 

\begin{prop}
    Let $\langle n\rangle$ and $\langle m\rangle$ be objects of $\Lambda$ and let 
    \[ \begin{tikzcd}
    U \colon &[-3em] \Hom_{\Lambda}(\langle n\rangle,\langle m\rangle)\arrow[r] & \on{Cyc}(\underline{n},\underline{m})
    \end{tikzcd} \]
    send $f \colon \langle m \rangle \rightarrow \langle n \rangle$ to its underlying map of sets. For $\varphi\in \Cyc(\underline{n},\underline{m})$, the fiber $U^{-1}(\varphi)$ is 
    \begin{enumerate}
		\item in bijection with $\ZZ_{n+1}$ if $\varphi$ is constant; and
		
		\item a singleton if $\varphi$ is non-constant. 
    \end{enumerate}
\end{prop}

A consequence of this proposition is another way to encode the data of a morphism in $\Lambda$, using slightly redundant information. 

\begin{cor}
    A morphism $\varphi \colon \langle n\rangle \to \langle m\rangle$ in $\Lambda$ is equivalently given by the data of a map $\underline{\varphi}\in \Cyc(\underline{n},\underline{m})$ together with a choice of linear order on each fiber of $\underline{\varphi}$ that is compatible with the cyclic order on $\langle n\rangle$.
\end{cor}

\section{An interval model and duality for the simplex category} \label{intervalmodelDelta}

We now turn to models based on marked subintervals and their relation to duality, starting with the simplex category in this section and then the cyclic category in the next section.  We mostly forgo proofs in these sections, as they are often analogous to those for the finite subset models.  

This approach gives another way to understand these categories geometrically, and it additionally provides a convenient way to understand the relationship between $\Delta$ and a subcategory $\nabla$ that is often called the \emph{interval category}.



\subsection{The banded simplex category} 

To start, we look at another geometric definition.

\begin{definition}
    A \emph{banded interval} is a pair of spaces $(I^\circ,J)$ with $J\subseteq I^\circ$, satisfying the following conditions.
    \begin{itemize}
		\item The subspace $J$ is homeomorphic to a (nonempty) disjoint union of interval spaces, 
		\[ J\cong \coprod_{i=0}^K J_i. \]
		
		\item Each of the components $J_i\subseteq J\subseteq I^\circ$ extends to an inclusion of the closed interval $I$ into $I^\circ$.  That is, the closure $\overline{J_i}$ of $J_i$ in $I$ is homeomorphic to a closed interval.  
		
		\item The closures of the $J_i$ are disjoint in $I^\circ$. 
    \end{itemize}
    A \emph{morphism of banded intervals} is a monotone homeomorphism $f \colon (I^\circ,J)\to (I^\circ, K)$ that is also a map of pairs. We denote the space of morphisms of banded circles by $\Homeo^+((I^\circ,J),(I^\circ,K))$. 
\end{definition} 

\begin{remark} \label{rmk:isomofpairs}
    The latter two technical-seeming conditions to guarantee that there is an isomorphism of pairs
    \[ (I^\circ \amalg_{\overline{J}} \pi_0(J),\pi_0(J)) \cong (I^\circ,\underline{n}), \]
    where $J$ has $n+1$ path components. It is this isomorphism that we use to define the relation between our two versions of $\Delta$. 
\end{remark}

\begin{definition}
    The \emph{banded simplex category} is the topological category $\Delta_{\Band}$ whose objects are banded intervals and whose mapping spaces are 
    \[ \Hom_{\Delta_{\Band}}((I^\circ,J),(I^\circ,K)):= \Homeo^+((I^\circ,J),(I^\circ,K)). \]
\end{definition}

If we choose, for every banded interval $(I^\circ,J)$, an identification 
\[ (I^\circ \amalg_{\overline{J}} \pi_0(J),\pi_0(J)) \cong (I^\circ,\underline{n}) \]
for $n+1=|\pi_0(J)|$ as in Remark \ref{rmk:isomofpairs}, then every monotone homeomorphism of banded intervals $f \colon (I^\circ,J)\to (I^\circ, K)$ induces a monotone homotopy equivalence of marked intervals 
\[ \begin{tikzcd}
    (I^\circ,\underline{n}) \arrow[r] & (I^\circ,\underline{m}). 
\end{tikzcd} \]
This construction defines a functor 
\[ \begin{tikzcd}
    P \colon &[-3em] \Delta_{\Band}\arrow[r] & \Delta. 
\end{tikzcd} \]
Pictorially, the action of the functor $P$ on morphisms can be visualized by Figure \ref{fig:Functor_P}.  
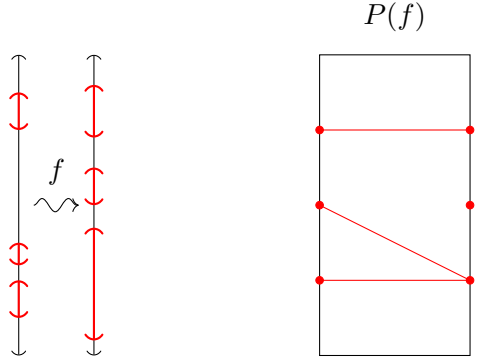
\begin{figure}[h!]
	\begin{tikzpicture}
		\draw[(-)] (0,-2) to (0,2);
		\draw[thick,red,(-)] (0,-1.5) to (0,-1);
		\draw[thick,red,(-)] (0,-0.8) to (0,-0.5);
		\draw[thick,red,(-)] (0,1) to (0,1.5);
		
		\draw[(-)] (1,-2) to (1,2);
		\draw[thick,red,(-)] (1,-1.8) to (1,-0.3);
		\draw[thick,red,(-)] (1,0) to (1,0.5);
		\draw[thick,red,(-)] (1,0.9) to (1,1.6);
		\begin{scope}[decoration=snake]
			\draw[->,decorate] (0.2,0) to node[label=above:$f$] {} (0.8,0);
		\end{scope}
		
		\draw (4,-2) rectangle (6,2);
		\foreach \y in {-1,0,1}{
		\draw[red,fill=red] (4,\y) circle (0.05); 
		\draw[red,fill=red] (6,\y) circle (0.05); 
		};
		\draw[red] (4,-1) to (6,-1);
		\draw[red] (4,0) to (6,-1);
		\draw[red] (4,1) to (6,1);
		\path (5,2.5) node {$P(f)$};
	\end{tikzpicture}
	\caption{A depiction of a morphism in $\Delta_{\on{Band}}$ and its image under $P$.} \label{fig:Functor_P}
\end{figure}

We can summarize this discussion with the following analogue to Proposition \ref{DeltaTopDK}.

\begin{prop}
    The functor $P \colon \Delta_{\Band} \rightarrow \Delta$ is a Dwyer-Kan equivalence.  
\end{prop}

\subsection{The interval category and duality} 

Our next objective is to define a subcategory of $\Delta$ that is described in a way similar to the banded simplex category.  As before, we begin with this topological description, and then give a combinatorial characterization. 

\begin{definition}
    An \emph{extremally banded interval} is a pair of spaces $(I^\circ,J)$ with $J\subseteq I^\circ$, satisfying the following conditions.
    \begin{itemize}
		\item The subspace $J$ is homeomorphic to a (nonempty) disjoint union of interval spaces, 
		\[ J\cong \coprod_{i=0}^k J_i. \]
		
		\item Each of the components $J_i\subseteq J \subseteq I^\circ\subseteq I$ extends to an inclusion of the closed interval $I$ into $I=\overline{I^\circ}$. That is, the closure $\overline{J_i}$ of $J_i$ in $I$ is homeomorphic to a closed interval in $[0,1]$.
		
		\item The closures of the $J_i$ are disjoint in $I^\circ$. 
		
		\item There are distinct components $J_i$ and $J_j$ of $J$ such that $0$ is a limit point of $J_i$ and $1$ is a limit point of $J_j$. 
	\end{itemize}
    A \emph{morphism of extremally banded intervals} is a monotone homeomorphism $f \colon (I^\circ,J)\to (I^\circ, K)$ that is also a map of pairs. We denote the space of morphisms of extremally banded circles by $\Homeo^+((I^\circ,J),(I^\circ,K))$. 
\end{definition}

Pictorially, an extremally banded interval looks like a banded interval, but with a subinterval that starts at the left, and a subinterval that ends at the right, as shown in Figure \ref{fig:ext_banded_interval}. 
\begin{figure}[h!]
    \begin{tikzpicture}
		\draw[(-)] (0,0) to (4,0);
		\begin{scope}[thick,red]
			\draw[(-)] (0,0) to (0.5,0);
			\draw[(-)] (1,0) to (1.7,0); 
			\draw[(-)] (2.3,0) to (4,0);
		\end{scope}
	\end{tikzpicture}
    \caption{An extremally banded interval, with the bands pictured in red.} \label{fig:ext_banded_interval} 
\end{figure}
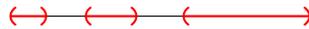

\begin{definition}
    The \emph{topological interval category} $\nabla_{\Band}$ has extremally banded intervals as objects and morphisms given by
    \[ \Hom_{\nabla_{\Band}}((I^\circ,J),(I^\circ,K))\cong \Homeo^+((I^\circ,J),(I^\circ,K)). \]
    The \emph{combinatorial interval category} is the subcategory $\nabla\subseteq \Delta$ with objects $[n]$ for $n\geq 1$ and morphisms $f \colon [n]\to [m]$ that \emph{preserve extremal elements}, in the sense that $f(0)=0$ and $f(n)=m$. 
\end{definition}

We can construct a functor 
\[ \begin{tikzcd}
    Q \colon &[-3em] \nabla_{\Band}\arrow[r] & \nabla 
\end{tikzcd} \] 
as follows. Given an extremally banded interval $(I^\circ,J)$, the orientation of $I$ gives a unique order-preserving identification $\{0,\ldots,n\}\cong \pi_0(J)$ where $n\geq 1$ by the last part of the previous definition. We then send $(I^\circ,J)$ to $[n]$.  We send a morphism $f \colon (I^\circ,J)\to (I^\circ,K)$ to the composite 
\[ \begin{tikzcd}
    {[n]} \arrow[r,"\cong"] & \pi_0(J) \arrow[r,"\pi_0(f)"] & \pi_0(K) \arrow[r,"\cong"] & {[m]}.
\end{tikzcd} \]

As with our topological model for $\Delta$, we do not lose any homotopical information by passing from the topological interval category $\nabla_{\Band}$ to its combinatorial counterpart $\nabla$, as the following proposition makes precise. 

\begin{prop} 
    The functor $Q$ is a Dwyer-Kan equivalence. 
\end{prop}



We now describe the duality between $\Delta$ and $\nabla$, using the tools of banded intervals. Note that if $(I^\circ,J)$ is a banded interval, then $(I^\circ, I^\circ\setminus J)$ is an extremally banded interval. Performing this construction twice returns the original banded interval. 

Additionally, if $f \colon I^\circ \to I^\circ$ is a monotone homeomorphism, then $f^{-1} \colon I^\circ\to I^\circ$ is as well. Moreover, for $f\in \Homeo^+((I^\circ,J),(I^\circ,K))$, we have $f^{-1}\in \Homeo^+((I^\circ,I^\circ\setminus J),(I^\circ,I^\circ\setminus K))$. Since taking inverses of isomorphisms is functorial, we thus obtain functors
\[ \begin{tikzcd}
    D_{\Delta} \colon &[-3em] \Delta_{\on{Band}} \arrow[r] & \nabla_{\Band}
\end{tikzcd} \]
and 
\[ \begin{tikzcd}
	\nabla_{\Band} \arrow[r] & \Delta_{\Band}
\end{tikzcd} \]
given by taking complements of subsets of $I^\circ$ and inverses of homeomorphisms. The following first duality result is immediate, since these processes are inverse to one another. 

\begin{prop}
    The functor 
    \[ \begin{tikzcd}
    D_{\Delta} \colon &[-3em] \Delta_{\Band} \arrow[r] & \nabla_{\Band}
    \end{tikzcd} \]
    is a Dwyer-Kan equivalence.
\end{prop}

We thus obtain the following corollary, which can be also be proved via direct combinatorial construction.

\begin{cor} 
    There is an equivalence of categories 
    \[ \begin{tikzcd}
    \Delta \arrow[r] & \nabla.
    \end{tikzcd} \]
\end{cor}

\section{An interval model and duality for the cyclic category} \label{intervalmodelLambda}

In this section, we translate the constructions from the previous section to the cyclic setting.  In this case, the dual category $\nabla$ is replaced by the opposite of the category $\Lambda$, which is simply $\Lambda$ itself, showing that the cyclic category is actually self-dual.  We use this approach to arrive at a more traditional combinatorial description of $\Lambda$ and discuss its canonical factorization.  

\subsection{The banded cyclic category and duality}

Following our banded interval model of $\Delta$, we now construct a an analogous topological model for $\Lambda$. 

\begin{definition}
    A \emph{banded circle} is a pair $(S^1,J)$ of topological spaces such that $J\subseteq S^1$ is homeomorphic to a disjoint union of finitely many interval spaces whose closures in $S^1$ are disjoint. 
	
    A \emph{monotone homeomorphism of banded circles} $f \colon (S^1,J)\to (S^1,K)$ is a monotone homeomorphism $f \colon S^1\to S^1$ such that $f(J)\subseteq K$. We denote by 
    \[ \Homeo^+((S^1,J), (S^1,K))\subseteq C^0(S^1,S^1) \] 
    the subspace of monotone homeomorphisms of banded circles. 
\end{definition} 

A schematic image of a banded circle is provided in Figure \ref{fig:Banded_circle}. 
\begin{figure}[h] 
	\begin{center}
		\begin{tikzpicture}
			\draw (0,0) circle (2); 
			\foreach \x/\y in {0/30, 100/110), 190/270}{
				
			};
			\draw[thick,red,(-|] (0:2) arc (0:30:2);
			\draw[thick,red,(-)] (100:2) arc (100:110:2);
			\draw[thick,red,|-|] (190:2) arc (190:270:2);
		\end{tikzpicture}
	\end{center}
	\caption{A banded circle with $J\subseteq S^1$ marked in red. Note that the intervals which make up $J$ can be open, closed, or half-open.} \label{fig:Banded_circle}
\end{figure}
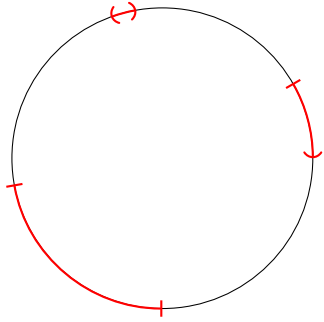  

\begin{definition}
    The \emph{banded cyclic category} $\Lambda_{\Band}$ has objects banded circles and spaces of morphisms given by $\Homeo^+((S^1,J),(S^1,K))$. 
\end{definition}

As with banded intervals, we can identify 
\[ \left(S^1\coprod_J \pi_0(J),\pi_0(J)\right) \cong (S^1,\underline{n}). \]
This construction yields a functor 
\[ \begin{tikzcd}
    \Lambda_{\Band}\arrow[r] & \Lambda 
\end{tikzcd} \]
as in the simplicial case. Once again, the process of passing from the topological definition to the combinatorial one does not lose homotopical information.

\begin{prop}  
    The functor 
    \[ \begin{tikzcd}
    \Lambda_{\Band}\arrow[r] & \Lambda 
    \end{tikzcd} \]
    is a Dwyer-Kan equivalence. 
\end{prop}



Now, we show that the exact same construction that yields the duality between $\nabla$ and $\Delta$ extends to the cyclic setting.  Given a banded circle $(S^1,J)$, we can define a new banded circle $(S^1,J^c)$, where $J^c:=S^1\setminus J$ is the complement of $J$ in $S^1$. Since the inverse of a monotone homeomorphism of banded circles $f \colon (S^1,J)\to (S^1,K)$ is a monotone homeomorphism $f^{-1} \colon (S^1,K^c)\to (S^1,J^c)$ of banded circles, this construction yields a functor 
\[ \begin{tikzcd}
    D_{\Lambda} \colon  &[-3em] \Lambda_{\Band}^{\op} \arrow[r] & \Lambda_{\Band}. 
\end{tikzcd} \] 
It follows from construction that $D^{\op}_\Lambda\circ D_\Lambda$ is the identity functor, and so we obtain the following result. 

\begin{prop}
    The functor 
    \[ \begin{tikzcd}
		D_{\Lambda} \colon  &[-3em] \Lambda_{\Band}^{\op} \arrow[r] & \Lambda_{\Band} 
    \end{tikzcd} \]
    is an isomorphism of topological categories. 
\end{prop}

As in the simplicial case, we obtain duality for the usual category $\Lambda$, but the statement is more striking in this case, since we obtain that $\Lambda$ is self-dual.

\begin{cor} 
    There is an isomorphism of categories 
    \[ \begin{tikzcd}
	\Lambda^{\op} \arrow[r] & \Lambda.
    \end{tikzcd} \]
\end{cor}

\subsection{Compactification} \label{subsubsec:comp} 

The topological approach from intervals also provides a way to view $\Delta$ as a subcategory of $\Lambda$. Let us denote the one-point compactification of $I^\circ$ by $(I^\circ)^\ast$ and choose a homeomorphism $\varphi \colon (I^\circ)^\ast\to S^1$ such that the composite 
\[ \begin{tikzcd}
    I^\circ \arrow[r,hookrightarrow] & (I^\circ)^\ast\arrow[r,"\varphi","\cong"'] & S^1 
\end{tikzcd} \]
is monotone. Given a banded interval $(I^\circ, J)$, the pair $(S^1,\varphi(J))$ is a banded circle. Similarly, given a morphism $f \colon (I^\circ,J)\to (I^\circ,K)$ of banded intervals, the induced map on 1-point compactifications 
\[ \begin{tikzcd}
    \iota(f) \colon &[-3em] (S^1,\varphi(J))\arrow[r] & (S^1, \varphi(K)) 
\end{tikzcd} \] 
is a morphism of banded circles. We thus obtain a functor 
\[ \begin{tikzcd}
    \iota \colon &[-3em] \Delta_{\Band}\arrow[r] & \Lambda_{\Band}. 
\end{tikzcd} \]

\begin{prop}
    The induced functor 
    \[ \begin{tikzcd}
    \iota \colon &[-3em] \Delta \arrow[r] & \Lambda
    \end{tikzcd} \]
    is faithful and bijective on objects. 
\end{prop}

In much the same way as with $\Delta$, we can construct a generators-and-relations presentation of the cyclic category $\Lambda$.  The proof is very similar to the one for the analogous presentation of $\Delta$, with the additional complication that one must first show that morphisms in $\Lambda$ factor uniquely as an automorphism followed by a morphism in the image of $\iota \colon \Delta\to \Lambda$. 


\begin{definition} \label{defn:cyclic_idents}
    The cyclic category $\Lambda$ is generated by the face maps $d_i \colon [n-1]\to [n]$ of $\Delta$, the degeneracy maps $s_i \colon [n+1]\to [n]$ in $\Delta$, and the cyclic shift-by-one automorphisms $\tau_n \colon [n]\to [n]$, subject to:
    \begin{enumerate}
		\item the \emph{simplicial identities}
		\begin{align*}
		  d_j\circ d_i&=d_i\circ d_{j-1} & i<j\\
		  s_j\circ s_i&=s_i\circ s_{j+1} & i\leq j\\
		  s_j\circ d_i&=
		  \begin{cases}
				d_i \circ s_{j-1} & i<j\\
				\id_{[n]} & i=j,j+1\\
				d_{i-1}\circ s_j & i>j+1;
		  \end{cases} 
		\end{align*}
        and
		
		\item the \emph{cyclic identities} 
		\begin{align*}
		  \tau\circ d_i & = 
		  \begin{cases}
				d_{i-1}\circ \tau_{n-1} & 1\leq i\leq n\\ 
				d_n & i=0
		  \end{cases} \\
			\tau_n\circ s_i&=\begin{cases}
				s_{i-1}\circ \tau_{n+1} & 1\leq i\leq n\\
				s_n \circ \tau^2_{n+1} & i=0
		  \end{cases}\\
		  \tau_n^{n+1}&=\id_{[n]}.
		\end{align*}
    \end{enumerate}
\end{definition}

\subsection{The canonical factorization} 

We now turn to a final property of the cyclic category that clarifies the way in which $\Lambda$ is a well-behaved extension of $\Delta$: the canonical factorization of morphisms in $\Lambda$ into automorphisms and morphisms in $\Delta$.  To this end, we first describe the one-point compactification functor $\iota$ in combinatorial terms. 

\begin{construction} \label{const:include_Delta_in_Lambda}
    An object $[n]$ of $\Delta$ uniquely corresponds to a morphism 
    \[ \begin{tikzcd}
		\langle n\rangle \arrow[r] & \langle 0\rangle, 
    \end{tikzcd} \]
    in $\Lambda$, where $[n]$ specifies the linear order on the fiber of the underlying map of sets. Consequently, we obtain a \emph{cyclic closure functor}  
    \[ \begin{tikzcd}[row sep=0em]
		\iota \colon &[-3em] \Delta \arrow[r] & \Lambda 
    \end{tikzcd} \]
    defined on objects by
    \[ \begin{tikzcd}[row sep=0em]		
    & {[n]} \arrow[r,mapsto] & \langle n\rangle.
    \end{tikzcd} \] 
    On morphisms, the underlying map of sets is uniquely determined, and the linear orders on the fibers are those induced by the linear order on the source. 
\end{construction}


\begin{lemma}
    For any $i,j\in \langle n\rangle$, there is a unique automorphism 
    \[ \begin{tikzcd}
		\varphi\colon &[-3em] \langle n\rangle \arrow[r] & \langle n\rangle 
    \end{tikzcd} \]
    such that $\varphi(i)=j$. 
\end{lemma}

\begin{proof}
    This result follows directly from the fact that an automorphism in $\Lambda$ is precisely a $\mathbb Z/(n+1)$-equivariant map $\langle n\rangle \to \langle n\rangle$.
\end{proof}

Using the inclusions constructed above, we can now prove one final key property of $\Lambda$. 

\begin{prop} \label{prop:Lambda_CF}
    The composition map 
    \[ \begin{tikzcd}[row sep=0em]
		\Hom_{\Delta}([n],[m])\times \Aut_{\Lambda}(\langle n \rangle) \arrow[r] & \Hom_{\Lambda}(\langle n\rangle,\langle m\rangle)\\
		(f,\psi) \arrow[r,mapsto] & \iota(f)\circ \psi 
    \end{tikzcd} \]
    is a bijection. That is, every morphism in $\Lambda$ factors uniquely into a morphism in $\Delta$ and an automorphism of its source in $\Lambda$.
\end{prop}



\begin{proof}
     We begin by fixing a linear order on the underlying set of $\langle m\rangle$.  The proposition may then be proved by examining the linear order induced on the underlying set of $\langle n\rangle$ by the linear orders on fibers encoded by a morphism $\varphi \colon \langle n\rangle \to \langle m\rangle$.
\end{proof}

\begin{definition}
    The unique factorization of a morphism $\varphi$ in $\Lambda$ into $f_\varphi \circ \psi_\varphi$, where $f_\varphi$ is a morphism in $\Delta$ and $\psi_\varphi$ is an automorphism in $\Lambda$, is called the \emph{canonical factorization}.
\end{definition}

\begin{remark}
    While the topological approach does not simplify the proof of the canonical factorization, it does provide good intuition for it. If we consider the morphism $\varphi \colon \langle 1\rangle \to \langle 2 \rangle$ from Example \ref{eg:include}, we can draw the canonical factorization as the gluing of two cylinders as in Figure \ref{fig:gluing_mapping_cylinders}. 
    \begin{figure}[h!]
	\includegraphics[width=0.7\textwidth]{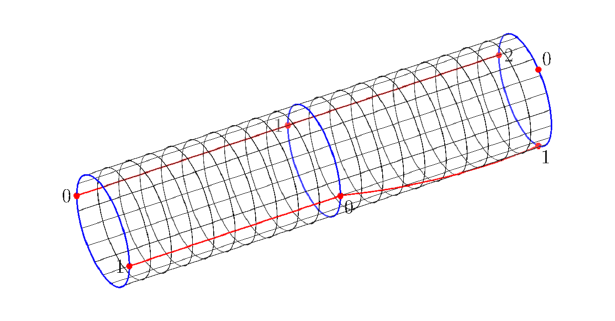}
		\caption{The gluing of two mapping cylinders.} \label{fig:gluing_mapping_cylinders}
    \end{figure} 
    Notice that the labeling of the points on the central circle by $0$ and $1$ is the only way to label them so that the right-hand morphism is the image of a map in $\Delta$.
\end{remark}

\section{Crossed simplicial groups and realizations} \label{csg}

Now that we have set up several different approaches to thinking about the category $\Lambda$ and its properties, we turn our attention to an axiomatization of some of these features, specifically the relationship with the simplex category $\Delta$ and the canonical factorization.

The more general notion we consider is that of \emph{crossed simplicial groups}, introduced by Krasauskas \cite{kras} under the name \emph{skew-simplicial groups}, and independently by Fiedorowicz and Loday \cite{fl}.  While we follow the terminology from the latter, our notation instead more closely follows that of \cite{dkcsg}.

\begin{definition}
    A \emph{crossed simplicial group} is a category $\Delta \mathfrak{G}$ satisfying the following conditions.
    \begin{enumerate}
		\item The simplex category $\Delta$ is a subcategory of $\Delta\mathfrak{G}$ that contains all the objects of $\Delta\mathfrak{G}$.
		
		\item Every morphism $\varphi \colon [n]\to [m]$ in $\Delta\mathfrak{G}$ factors uniquely as $\psi\circ g$, where $\psi$ is a morphism in $\Delta$ and $g\in \Aut_{\Delta \mathfrak{G}}([n])$. We call it the \emph{canonical factorization} of $\varphi$. 
    \end{enumerate}
    We denote the automorphism groups $\Aut_{\Delta\mathfrak{G}}([n])$ of $\Delta\mathfrak{G}$ by $\mathfrak{G}_n$. 
\end{definition}

\begin{remark}
    Notice that this definition is suppressing a subtlety: in most cases we technically consider categories $\Delta\mathfrak{G}$ in which the subcategory from condition (1) is isomorphic, but not equal, to $\Delta$. In practice, this distinction does not alter the way in which we work with crossed simplicial groups.
\end{remark}

\begin{remark}\label{rmk:var_defns_csg}
    There are two ways of defining crossed simplicial groups in the literature, either as categories $\Delta\mathfrak{G}$ as in \cite[Definition 1.1]{fl}, or as simplicial sets $\mathfrak{G}_\ast$ as in Construction \ref{const:sset_of_csg} below, following \cite[Definition 1.3]{kras}). These approaches are equivalent, as shown in \cite[Theorem 1.4]{kras}. In this paper, we have chosen to define crossed simplicial groups as categories, since our main objects of study are functors out of the category $\Lambda$. 
\end{remark}

\begin{example}
    The simplest example of a crossed simplicial group is the simplex category $\Delta$ itself.  In this case, the automorphism groups of the objects are trivial, so the canonical factorization does not give any new information.  
\end{example}

\begin{example}
    There is a category $\Delta\mathfrak{S}$, called the \emph{symmetric crossed simplicial group}, whose objects are the sets $[n]$ for $n\geq 0$, and whose morphisms are maps of sets  $f \colon [m]\to [n]$ together with a choice of linear orders on the fibers $f^{-1}(i)$ for $i\in [n]$. The corresponding automorphism groups are precisely the symmetric groups. This crossed simplicial group first appears, independently, in  \cite[Prop. 1.5]{kras} and \cite[Prop. 3.4]{fl}. The explicit description given here is from \cite[\S I.2]{dkcsg}.
\end{example}

\begin{example}
    The \emph{paracyclic category} $\Lambda_\infty$ is a category whose objects are  $\langle n\rangle$ for $n\geq 0$. A morphism $f \colon \langle n\rangle \to \langle m\rangle$ in $\Lambda$ is a non-decreasing map of sets $f \colon \ZZ\to \ZZ$ such that 
    \[ f(i+n+1)=f(i)+m+1 \]
    for all $i\in \ZZ$. The paracyclic category is a crossed simplicial group, all of whose automorphism groups are isomorphic to $\ZZ$. This crossed simplicial group first appears in \cite[Example 3]{fl}. The term \emph{paracyclic}, as well as the explict description given here, comes from \cite{gj}.
\end{example}

\begin{construction} \label{const:sset_of_csg}
    Let $\Delta\mathfrak{G}$ be a crossed simplicial group. Given a morphism $\varphi \colon [n]\to [m]$ in $\Delta$, we can define a map of sets $\varphi^\ast \colon \mathfrak{G}_m\to \mathfrak{G}_n$ as follows.  Given $g\in \mathfrak{G}_m$, the canonical factorization allows us to complete the diagram 
    \[ \begin{tikzcd}
		& {[m]}\arrow[d,"\varphi"]\\
		{[n]}\arrow[r,"g"'] & {[n]}
    \end{tikzcd} \]
    uniquely to a commutative diagram 
    \[ \begin{tikzcd}
		{[m]}\arrow[r,"{\varphi^\ast(g)}"]\arrow[d,"\psi"']& {[m]}\arrow[d,"\varphi"]\\
		{[n]}\arrow[r,"g"'] & {[n],}
    \end{tikzcd} \]
    where $\psi$ is a morphism in $\Delta$. The uniqueness of the canonical factorization implies that the maps $\varphi^\ast$ assemble into a simplicial set 
    \[ \begin{tikzcd}[row sep=0em]
		\mathfrak{G}_\ast \colon &[-3em] \Deltaop \arrow[r] & \Set \\
		& {[n]}\arrow[r,mapsto] & \mathfrak{G}_n.
    \end{tikzcd} \]
\end{construction}

\begin{remark}
    The simplical set $\mathfrak{G}_\ast$, together with the group structures on $\mathfrak{G}_n$ and a few algebraic identities, suffices to determine the category $\Delta\mathfrak{G}$ completely. While this result is not important to our exposition here, proofs of it can be found in \cite[Prop. 1.6]{fl} and \cite[Theorem 1.4]{kras}.
\end{remark}

\begin{remark} 
    The notation for a crossed simplicial group originally used in \cite{fl} was $\Delta G$, which is also the notation used in \cite[Ch.\ 6]{loday}. In the latter, the cyclic category $\Lambda$ is denoted by $\Delta C$. We have chosen to instead use the notation $\Delta \mathfrak{G}$ from \cite{dkcsg}, since it helps distinguish the simplicial sets associated to crossed simplicial groups from other notation in use here.
\end{remark}

\begin{notation} \label{notation:Lambda_star}
    Because we use the conventional notation $\Lambda$ for the cyclic category, rather than the crossed simplicial group notation $\Delta \mathfrak G$, we must fix a notation for the simplicial set associated to $\Lambda$. To adhere as closely to our convention as possible, we write $\Lambda_\ast$ for this simplicial set. We then have 
	\begin{align*}
		\Lambda_0 &= \{\id_{[0]}\}\\
		\Lambda_1 &= \{\id_{[1]},\tau_1\}\\
		\Lambda_2 &= \{\id_{[2]},\tau_2,\tau_2^2\}\\
		\vdots & \qquad \vdots\text{.}
	\end{align*} 
    As we show in Lemma \ref{lem:degnsimp} below, the only non-degenerate simplices of $\Lambda_\ast$ are $\id_{[0]}$ and $\tau_1$. 
\end{notation}

There is a profound relationship between the cyclic category, cyclic objects, and circle actions, which was an initial motivation for the definition of the cyclic category by Connes \cite{connes}. This relationship extends to other crossed simplicial groups as well, and may be summarized in the following theorem.  The first and third statements are proved in \cite[Thm 5.3]{fl}, and the second statement follows from \cite[Thm 5.12]{fl}. 

\begin{theorem}
    Let $\Delta\mathfrak{G}$ be a crossed simplicial group, and let $\mathfrak{G}_\ast$ be the associated simplicial set. 
    \begin{enumerate}
		\item The realization $|\mathfrak{G}_\ast|$ is canonically equipped with the structure of a topological group. 
		
		\item There is a homotopy equivalence $B|\mathfrak{G}_\ast|\simeq |N(\Delta \mathfrak{G})|$, where $B|\mathfrak{G}_\ast|$ denotes the classifying space of the topological group $|\mathfrak{G}_\ast|$. 
		
		\item Given a functor $X \colon \Delta\mathfrak{G}^{\op} \to \Set$, the realization $|X|$ of the underlying simplicial object comes equipped with a continuous action of $|\mathfrak{G}_\ast|$, providing a functor 
		\[ \begin{tikzcd}		  \Set_{\Delta\mathfrak{G}}\arrow[r] & {|\mathfrak{G}_\ast| \text{-}\Top}
		\end{tikzcd} \] 
		to spaces equipped with a continuous $|\mathfrak{G}_\ast|$-action. 
	\end{enumerate}
\end{theorem}

In the case of the cyclic category $\Lambda$, the corresponding topological group is $S^1\cong U(1)\cong SO(2)$, a fact we want to demonstrate on the level of topological spaces.  First, we prove the following lemma that we mentioned above.

\begin{lemma} \label{lem:degnsimp}
    For $n\geq 2$, every $n$-simplex of $\Lambda_\ast$ is degenerate.
\end{lemma}

\begin{proof}
    Let $s_i \colon [n+1]\to [n]$ be the $i$\textsuperscript{th} degeneracy map in $\Delta$. A short computation shows that $s_n^\ast(\tau_n^k)=\tau_{n+1}^k$ for $0\leq k\leq n$. Moreover, $s_{n-1}^\ast(\tau_n^{n})=\tau_{n+1}^{n+1}$, completing the proof. 
\end{proof}

\begin{prop}\label{prop:realization_Lambda_star}
    There is a homeomorphism 
    \[ |\Lambda_\ast|\cong S^1. \]
\end{prop}

\begin{proof}
    Expressing $|\Lambda_\ast|$ in terms of non-degenerate simplices, we see that it is obtained by gluing both endpoints of a single 1-simplex to a single 0-simplex.
\end{proof}

We do not construct the group structure on $|\Lambda_\ast|$ here, instead directing the interested reader to \cite[\S 5]{fl} or \cite[Ch.\ 7]{loday}.

\section{Cyclic sets and their model structure} \label{cyclicsetmc}


We now turn to considering the categories $\SSets$ of simplicial sets and $c\Sets$ of cyclic sets.  The category $\SSets$ of simplicial sets has as objects the functors $\Deltaop \rightarrow \Sets$ and as morphisms the natural transformations between them.  This category has a model structure in which the weak equivalences are the weak homotopy equivalences, or maps that induce isomorphisms on all homotopy groups, the fibrations are the Serre fibrations, and the cofibrations are the retracts of cell inclusions \cite{quillen}.  In what follows, we simply denote this model structure by $\SSets$.

Similarly, a cyclic set is given by a functor $X \colon \Lambdaop \rightarrow \Sets$ and is thus a simplicial set together with bijections $t_n \colon X_n\to X_n$ for every $n$ satisfying the duals of the identities given in Definition \ref{defn:cyclic_idents}.  The category $c\Sets$ of cyclic sets has a model structure that was originally established by Dwyer, Hopkins, and Kan \cite{dhk}.  In this section, we give a full exposition of this model structure.

We first observe that the categories $\SSets$ and $c\Sets$ are connected by the forgetful functor $j^* \colon c\Sets \rightarrow \SSets$ induced by the inclusion functor $j \colon \Delta \rightarrow \Lambda$.  As a first step, let us consider the application of the functor $j^*$ to representables.

In the category of simplicial sets, for for every integer $n\geq 0$, there is a representable simplicial set $\Delta[n]$ given by
\[ \Hom_\Delta(-,[n]) \colon \Deltaop \rightarrow \Sets. \]
Similarly, for every integer $n\geq 0$, there is a representable cyclic set $\Lambda[n]$, given by 
\[ \Hom_{\Lambda}(-, \langle n \rangle) \colon \Lambda^{\op} \rightarrow \Sets. \]
By the Yoneda Lemma, these cyclic sets have the universal property that the elements $x\in X_n$ correspond uniquely to maps $c_x \colon \Lambda[n] \to X$ such that $c_x(\id_{\langle n\rangle})=x$. This universal property can be encoded as a natural isomorphism 
\[ \Hom_{c\Set}(\Lambda[n],X)\cong X_n. \]
Since the simplicial representables $\Delta[n]$ have an analogous universal property in $\SSets$, there is a composite isomorphism 
\[  \Hom_{c\Sets}(\Lambda[n], X)\cong X_n \cong \Hom_{\SSets}(\Delta[n], j^*X). \] 

Because the functor $j^*$ preserves colimits, composing it with geometric realization provides us with a colimit-preserving functor $|j^\ast(-)|$ producing a topological space from a cyclic set.  The key to understanding the realization of cyclic sets lies in the realization of the representables $|\Lambda[n]|$. 

\begin{prop}\label{prop:cyc_reps_circle_bundles} 
    The simplicial object
    \[ \begin{tikzcd}
		\Delta^{\op}\arrow[r,"{j^\ast \Lambda[-]}"] & \SSets \arrow[r,"{| - |}"] & \Top 
	\end{tikzcd} \]
    is naturally isomorphic to $S^1\times |\Delta[-]|$.
\end{prop}

\begin{proof}
    We construct a map between the two, following \cite[Prop.\ 5.1]{fl}, and leave it to the reader to verify it is natural and a homeomorphism. A consequence of the canonical factorization from Proposition \ref{prop:Lambda_CF} is that we can write each $k$-simplex of $\Lambda[n]$ uniquely as a pair $(\varphi,g)$, where $\varphi \colon [k]\to [n]$ is a morphism in $\Delta$ and $g$ is an automorphism of $\langle k\rangle $ in $\Lambda$. 
	
    Recall from Notation \ref{notation:Lambda_star} that $\Lambda_\ast$ denotes the simplicial set associated to $\Lambda$, and from Proposition \ref{prop:realization_Lambda_star} that the realization of this simplicial set is homeomorphic to the circle.  Explicitly checking the action of morphisms in $\Delta$ shows that the projection $\Lambda[n]\to \Lambda_\ast$ given by $(\varphi,g) \mapsto g$ is a simplicial map. On the other hand, we can define a continuous map 
    \[ \begin{tikzcd}[row sep=0em]
    {|\Lambda[n]|} \arrow[r] & {|\Delta[n]|}\\
    {[(\varphi,g),x]} \arrow[r,mapsto] & {[\varphi,g\cdot x]}
    \end{tikzcd} \]
    where the automorphisms of $\langle k\rangle$ act on $|\Delta[k]|$ by permuting the vertices. Together, these two maps define the desired continuous maps $|\Lambda[n]|\to |\Lambda_\ast| \times |\Delta[n]|\cong S^1\times |\Delta[n]|$. See Figure \ref{fig:Lambda_2} for a depiction of $j^\ast\Lambda[2]$.
\end{proof}

\begin{figure}[h!]
    \begin{tikzpicture}[decoration={markings,mark=at position 0.5 with \arrow{>}}]
		\path[fill=black] (0,0) circle (0.05);
		\path (0,0) node[label=below:$0$] {};
		\path[fill=black] (4,0) circle (0.05);
		\path (4,0) node[label=below:$1$] {};
		\draw[postaction={decorate}] (0,0) to (4,0);
		\draw[postaction=decorate] (0,0) to[out=150,in=180] (0,1.5) to[out=0,in=30] (0,0);
		\draw[dashed,postaction=decorate] (4,0) to[out=150,in=180] (4,1.5) to[out=0,in=30] (4,0);
		\draw (4,1.5) to[out=0,in=30] (4,0);
		\draw[dashed,postaction={decorate}] (0,0) to[out=45,in=180] (2,1.5) to[out=0,in=135] (4,0);
		\draw (2,1.5) to[out=0,in=135] (4,0);
		\path[fill=blue,opacity=0.2] (0,0) to[out=150,in=180] (0,1.5) to (2,1.5) to[out=180,in=45] (0,0) -- cycle;
		\path[fill=red,opacity=0.2] (0,0) to[out=45,in=180] (2,1.5) to (4,1.5) to[out=180,in=150] (4,0) --cycle;
		\path[fill=blue,opacity=0.2]  (0,0) to (4,0) to[out=135,in=0] (2,1.5) to (0,1.5) to[out=0,in=30] (0,0) --cycle;
		\path[fill=red,opacity=0.2] (2,1.5) to[out=0,in=135] (4,0) to[out=30,in=0] (4,1.5)  --cycle;
	\end{tikzpicture}
	\caption{A depiction of the simplicial set $j^\ast\Lambda[1]$ with the two non-degenerate 2-simplices pictured in red and blue.} \label{fig:Lambda_2}
\end{figure}
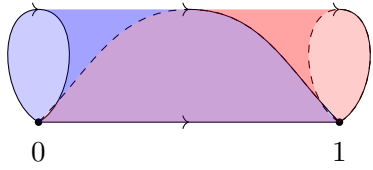

Since every functor from a small category to $\Set$ is a colimit of representables, and realization functors preserve colimits, we obtain two formulas for the realization of a cyclic set $X$. The first is the more common formula in terms of the underlying simplicial set:
\[ |j^\ast X|\cong \colim\limits_{\Delta[-]\to j^\ast X} |\Delta[-]|,
\]
where the colimit is taken over the overcategory $\Delta/X$. The second is similar in form:
\[ |j^\ast X|\cong \colim_{\Lambda[-]\to X} |j^\ast\Lambda[-]|,
\]
but the colimit is taken instead over the overcategory $\Lambda/X$.

These two ways of computing $|j^\ast X|$ suggest two different ways of thinking about the elements of $X_n$. On the one hand, we can think of $x\in X_n$ as an $n$-simplex of the underlying simplicial set. On the other, we can think of $x\in X_n$ as something we might call an ``{$n$-cyclex}": a map from $\Lambda[n]\to X$, or a specified copy of $S^1\times \Delta[n]$ in the realization. One of the key distinctions between these two perspectives is that there can be multiple ``cyclices" that are isomorphic, but not equal.  This feature is a reflection of the fact that $\Lambda$ admits non-trivial automorphisms, while $\Delta$ does not. 

Returning to our discussion of realizations of cyclic sets, Proposition \ref{prop:cyc_reps_circle_bundles} has an immediate corollary. 

\begin{cor} \label{dhk2.2} 
    The induced functor $j^*\Lambda[-] \colon \Lambda \rightarrow \SSets$ sends every morphism in $\Lambda$ to a weak equivalence. 
\end{cor}

With this corollary in place, we can now explore the homotopy theory of cyclic sets, as first described by Dwyer, Hopkins, and Kan in \cite[3.1]{dhk}.  We give a new proof here, using a recognition theorem establishing that the model structure for cyclic sets is cofibrantly generated, a fact that is implicit but not emphasized in the original proof.  

\begin{theorem} \label{dhk3.1} 
    The category $c\Sets$ of cyclic sets admits a cofibrantly generated model structure in which a map $X \rightarrow Y$ is:
    \begin{itemize}
        \item a weak equivalence or a fibration if the induced map $j^*X \rightarrow j^*Y$ is a weak equivalence or fibration, respectively, in $\SSets$; and 

        \item a cofibration if it is a retract of a (possibly transfinite) composite of cobase extensions along the inclusions $\partial \Lambda[n] \rightarrow \Lambda[n]$ for $n \geq 0$. 
    \end{itemize} 
\end{theorem}

\begin{proof}
    We claim that we can take as generating cofibrations the set
    \[ I = \{\partial \Lambda[n] \rightarrow \Lambda[n] \mid n \geq 0 \} \]
    and as generating acyclic cofibrations the set
    \[ J = \{\Lambda[n,k] \rightarrow \Lambda[n] \mid n \geq 1, 0 \leq k \leq n\}. \]
    Using these sets, we want to verify the conditions of \cite[11.3.1]{hirsch}.  
	
    First, observe that the category $c\Sets$ has all limits and colimits, and that the weak equivalences are closed under retracts and satisfy the two-out-of-three property.  It is also not hard to check that the sets $I$ and $J$ satisfy the small object argument, establishing condition (1) of \cite[11.3.1]{hirsch}.
	
    One can check from the definitions of fibration and acyclic fibration, or see \cite[3.2, 3.3]{dhk}, that a map in $c\Sets$ is a fibration if and only if it has the right lifting property with respect to the maps in $J$, and that a map is an acyclic fibration if and only if it has the right lifting property with respect to the maps in $I$.  It follows immediately that a map is an $I$-injective if and only if it is a $J$-injective and a weak equivalence, establishing conditions (2) and (4)(b) of \cite[11.3.1]{hirsch}.  
	
    It remains to show that every $J$-cofibration is an $I$-cofibration and a weak equivalence.  If a map $A \rightarrow B$ is a $J$-cofibration, then it has the left lifting property with respect to the fibrations, using the results from the previous paragraph.  In particular, it has the left lifting property with respect to the acyclic fibrations, from which it follows, again by the results of the previous paragraph, that it is an $I$-cofibration.  It remains to show that $A \rightarrow B$ is a weak equivalence.  
	
    Denote the left and right adjoints of $j^\ast$ by $j_\ast$ and $j_!$ respectively, and note that  $j_\ast \Delta[n]\cong \Lambda[n]$. Consider a horn inclusion $\Lambda^k[n]\to \Delta[n]$. Writing each of these two simplicial sets as a colimit over its category of non-degenerate simplices, and applying $j^\ast \circ j_\ast$, we see that the morphism
    \[  \begin{tikzcd}
    j^\ast (j_\ast \Lambda^k[n]) \arrow[r] & j^\ast(j_\ast \Delta[n]) 
    \end{tikzcd} \]
    is isomorphic to the morphism 
    \[ \begin{tikzcd}
    \colim_{(\Delta \downarrow \Lambda^k[n])_{ND}} j^\ast \Lambda[r] \arrow[r] & \colim_{(\Delta \downarrow \Delta[n])_{ND}} j^\ast \Lambda[r] 
    \end{tikzcd} \]
    induced by the inclusion of categories of simplices. Applying the geometric realization and using Proposition \ref{prop:cyc_reps_circle_bundles} yields the morphism 
    \[ \begin{tikzcd}
    \colim_{(\Delta \downarrow \Lambda^k[n])_{ND}} S^1\times \Delta^r  \arrow[r] & \colim_{(\Delta \downarrow \Delta[n])_{ND}} S^1\times \Delta^r.  
    \end{tikzcd} \]
    However, this morphism is isomorphic to $S^1\times |\Lambda^k[n]|\to S^1\times \Delta^n$, which is a homotopy equivalence. As such we see that $j^\ast(j_\ast\Lambda^k[n]) \to j^\ast(j_\ast \Delta[n])$ is a weak equivalence. It is not hard to check that this morphism is also a monomorphism, and thus is an acyclic cofibration in $\SSets$. By adjointness, it follows that if $\pi \colon X\to Y$ is a fibration in $\SSets$, then the induced map $j^\ast(j_!X)\to j^\ast(j_! Y)$ is also a fibration.  Now suppose that $f\colon A\to B$ is a $J$-cofibration, and let $\pi \colon X\to Y$ be a fibration in $\SSets$. A lifting problem 
    \[ \begin{tikzcd}
    j^\ast A \arrow[r]\arrow[d,"f"'] & X \arrow[d,"\pi"]\\
    j^\ast B \arrow[ur,dashed]\arrow[r] & Y 
    \end{tikzcd} \]
    is equivalent to a lifting problem
    \[ \begin{tikzcd}
	A \arrow[r]\arrow[d,"f"'] & j_!X \arrow[d,"j_!(\pi)"]\\
		B \arrow[r]\arrow[ur,dashed] & j_! Y.
    \end{tikzcd} \]
    However, by the preceding paragraph, the morphism $j_!(\pi)$ is a fibration, and so this lifting problem has a solution. Thus, $j^\ast A\to j^\ast B$ is an acyclic cofibration of simplicial sets, so that $A\to B$ is a weak equivalence as desired.
\end{proof}

Some of the motivation for this model structure is its comparison with the following model structure on $S^1$-spaces, or topological spaces equipped with an $S^1$-action.  We denote by $\Top^{S^1}$ the category of $S^1$ spaces with equivariant continuous maps. 

\begin{theorem} \cite[4.1]{dhk} \label{dhk4.1} 
    The category $\Top^{S^1}$ of $S^1$-spaces admits a model structure in which a map $X \rightarrow Y$ is:
    \begin{itemize}
        \item a weak equivalence or a fibration whenever the underlying map of topological spaces $X \rightarrow Y$ is a weak homotopy equivalence or a Serre fibration, respectively; and

        \item a cofibration if it is a retract of a (possibly transfinite) composite of cobase extensions along the inclusions $S^1 \times |\partial \Delta[n]| \rightarrow S^1 \times |\Delta[n]|$ for $n \geq 0$.
    \end{itemize} 
\end{theorem}

To compare these two model structures, it is necessary to promote the realization of cyclic sets from a functor 
\[ \begin{tikzcd}
	{|j^\ast(-)|} \colon &[-3em] c\Set \arrow[r] & \Top
\end{tikzcd} \]  
to a functor 
\[ \begin{tikzcd}
    L^c \colon &[-3em] c\Set \arrow[r] & \Top^{S^1}.
\end{tikzcd} \]
The construction of $L^c$ turns out to be purely formal. Since colimits in $\Top^{S^1}$ are preserved by the functor which sends each $S^1$-space to its underlying topological space, we can assign 
\[ L^c(X):=\colim_{\Lambda[-]\to X} |j^\ast\Lambda[-]| \]
where the colimit is taken in $\Top^{S^1}$ and the realizations $|j^\ast\Lambda[n]|\cong S^1\times |\Delta[n]|$ are equipped with the canonical $S^1$-action on the first factor. We thus obtain the following result. 

\begin{prop} \label{dhk2.8} \cite[2.8]{dhk}
    There is a functor $L^c \colon c\Sets \rightarrow \Top^{S^1}$ such that the diagram
    \[ \xymatrix{c\Sets \ar[r]^{L^c} \ar[dr]_{|j^*|} & \Top^{S^1} \ar[d] \\
    & \Top} \]
    commutes up to natural equivalence.
\end{prop} 

With the functor $L^c$ in place, we can complete our discussion of the homotopy theory of cyclic sets by providing the comparison between the two model structures. 

\begin{theorem} \label{dhk4.2} \cite[4.2]{dhk}
    The functor $L^c \colon c\Sets \rightarrow \Top^{S^1}$ has as right adjoint the functor $R^c = \Hom(L^c \Lambda[-], -) \colon \Top^{S^1} \rightarrow c\Sets$.  Moreover, this pair of adjoint functors is a Quillen equivalence of model categories.
\end{theorem}

\section{Levelwise model structures on cyclic spaces} \label{levelwisemcs}

In this section, we consider functors $X \colon \Lambdaop \rightarrow \SSets$, and we use the standard model structure on simplicial sets \cite{quillen} to produce model structures on the category $\SSets^{\Lambdaop}$ of all such functors.  As is usual for categories of functors from a small category, we have: 
\begin{itemize}
    \item the \emph{projective model structure}, in which weak equivalences and fibrations are given levelwise \cite[11.6.1]{hirsch}; and 
    
    \item the \emph{injective model structure}, in which weak equivalences and cofibrations are given levelwise \cite{heller}, \cite[A.2.8]{lurie}.
\end{itemize}
However, there is another model structure we can also consider.

One of the nice features of the category $\Deltaop$ is that it has the structure of a Reedy category \cite{reedy}, and as a consequence the category $\SSets^{\Deltaop}$ of simplicial spaces can be given the Reedy model structure.  Indeed this situation is particularly nice in that the Reedy model structure agrees with the injective structure \cite[15.8.7]{hirsch}.

The category $\Lambda$ is not a Reedy category, since its objects have nontrivial automorphisms, but it does have the structure of a generalized Reedy category in the sense of Berger and Moerdijk \cite{bm}, which we now define.  Recall that a subcategory $\mathcal D$ of a category $\mathcal C$ is \emph{wide} if it contains all the objects of $\mathcal C$.  As an example that plays a role in this definition, we denote by $\Iso(\mathcal C)$ the maximal subgroupoid of $\mathcal C$.

\begin{definition} \cite[1.1]{bm}
Let $\mathcal C$ be a small category.  A \emph{generalized Reedy structure} on $\mathcal C$ consists of wide subcategories $\mathcal C^+$ and $\mathcal C^-$ and a degree function $d \colon \ob(\mathcal C) \rightarrow \mathbb N$ such that:
\begin{enumerate}
    \item noninvertible morphisms in $\mathcal C^+$ raise the degree, while those in $\mathcal C^-$ lower degree, and isomorphisms in $\mathcal C$ preserve the degree;
    
    \item $\mathcal C^+ \cap \mathcal C^- = \Iso(C)$;
    
    \item every morphism $f$ of $\mathcal C$ can be factored as $f=gh$ with $g$ in $\mathcal C^+$ and $h$ in $\mathcal C^-$, and this factorization is unique up to isomorphism; and
    
    \item if $f$ is a morphism of $\mathcal C^-$ and $f \theta = f$ for some isomorphism $\theta$ in $\mathcal C$, then $\theta$ is an identity map. 
\end{enumerate}
\end{definition}



In the general context of crossed simplicial groups, Berger and Moerdijk show in \cite[2.7]{bm} that $\Lambda$ has the structure of a generalized Reedy category. We give a more direct description here. We take the degree function $d \colon \ob(\Lambda) \rightarrow \mathbb N$ to be given by $\langle n \rangle \mapsto n$; then we can take $\Lambda^+$ to be the wide subcategory whose morphisms preserve or increase degree, and likewise $\Lambda^-$ the wide subcategory whose morphisms preserve or lower degree.  

\begin{remark}
    In \cite[\S 2]{bm}, Berger and Moerdijk introduce a notion of \emph{crossed groups} over a Reedy category that generalize the crossed simplicial groups as we have described above. What we call a crossed simplicial group $\Delta\mathfrak{G}$ is, in their nomenclature, the \emph{total category} associated to the \emph{crossed $\Delta$-group} $\mathfrak{G}_\ast$. As in Remark \ref{rmk:var_defns_csg}, this distinction reflects two different but equivalent ways of defining crossed simplicial groups.
\end{remark}

Let $\Lambda^+(n)$ be the category whose objects are the non-invertible morphisms in $\Lambda^+$ with codomain $\langle n \rangle$ and in which a morphism from $u \colon s \rightarrow r$ to $u' \colon s' \rightarrow r$ is given by a morphism $w \colon s \rightarrow s'$ such that $u=u'w$.  The automorphism group $\Aut(\langle n \rangle) =C_{n+1}$ acts on $\Lambda^+(n)$ by composition.  For each functor $X \colon \Lambdaop \rightarrow \SSets$ and each $n \geq 0$, the $n$-\emph{th latching object} $L_n(X)$ of $X$ is defined to be 
\[ L_n(X) = \underset{\langle m \rangle \to \langle r \rangle}{\colim} X_s \]
with colimit taken over $\Lambda^+(n)$. Observe that $\Aut(\langle n \rangle) = C_{n+1}$ acts on $L_n(X)$.

Dually, we can define the category $\Lambda^-(n)$ with objects the non-invertible morphisms in $\Lambda^-$ with domain $\langle n \rangle$ and morphisms defined analogously.  Again, the automorphism group $\Aut(\langle n \rangle) = C_{n+1}$ acts on $\Lambda^-(n)$ by precomposition.  Then given a functor $X \colon \Lambdaop \rightarrow \SSets$ the $n$-\emph{th matching object} $M_n(X)$ of $X$ is defined to be 
\[ M_n(X) = \lim_{\langle n \rangle \to \langle m \rangle} X_s \]
with limit taken over the category $\Lambda^-(n)$, and observe that it has an action of $C_{n+1}$.  

Just as for ordinary categories, the natural maps $L_n(X) \rightarrow X_n \rightarrow M_n(X)$ induce \emph{relative latching maps}
\[ X_n \cup_{L_n(X)} L_n(Y) \rightarrow Y_n \]
and \emph{relative matching maps}
\[ X_n \rightarrow M_n(X) \times_{M_n(Y)} Y_n \]
for all $n \geq 0$.

Now, consider the category of $C_{n+1}$-simplicial sets, which we regard as the category of functors $C_{n+1} \rightarrow \SSets$, where $C_{n+1}$ is treated as a category with a single object.  We can equip this category with the projective model structure, in which weak equivalences and fibrations are given levelwise.  Using the discussion above, we can think of the relative latching maps and relative matching maps as morphisms in this category.

The following theorem is the specialization of \cite[1.6]{bm} to the case of cyclic spaces; we refer to the model structure described here as the \emph{generalized Reedy model structure}.  Observe that, in analogy with simplicial spaces, we want to take functors out of $\Lambdaop$, we simplify the notation in light of the isomorphism between $\Lambda$ and $\Lambdaop$ here and in what follows.

\begin{theorem}
There is a model structure on the category $\SSets^{\Lambda}$ in which a map $f \colon X \rightarrow Y$ is:
\begin{enumerate}
    \item a weak equivalence if for each $n \geq 0$, the induced map $f_n \colon X_n \rightarrow Y_n$ is a weak equivalence in $\SSets^{C_{n+1}}$;
    
    \item a cofibration if for each $n \geq 0$, the relative latching map
    \[ X_n \cup_{L_nX} L_nY \rightarrow Y_n \]
    is a cofibration in $\SSets^{C_{n+1}}$; and 
    
    \item a fibration if for each $n \geq 0$, the relative matching map
    \[ X_n \rightarrow M_nX \times_{M_nY} Y_n \]
    is a fibration on $\SSets^{C_{n+1}}$.
\end{enumerate}
\end{theorem}

Note that, following Berger and Moerdijk's result, we can replace the category $\SSets$ with any $\mathcal R$-\emph{projective model category}, or model category $\mathcal E$ such that for each $n \geq 0$, the category $\mathcal E^{C_n}$ admits the projective model structure.

Further observe that, unlike for simplicial spaces, the generalized Reedy structure on $\SSets^\Lambda$ is not the same as the injective model structure.  The cofibrations described above are injective cofibrations, but not conversely.  Thus, we have three distinct model structures on $\SSets^\Lambda$ with levelwise fibrations: injective, generalized Reedy, and projective.

Let us now examine the generalized Reedy structure in more detail.  In the spirit of similar results for simplicial spaces, we want to describe the matching objects more explicitly, enabling us to give a description of the generating cofibrations and generating acyclic cofibrations in this model structure.

Let $X$ be a cyclic space and $n \geq 0$.  By definition, we have 
\[ M_n X = \lim_{\Lambda[n] \to \Lambda[m]} X_m, \]
where the limit is taken over the category $\Lambda^-(n)$.  Using the representability of $\Lambda[m]$ and the duality of the category $\Lambda$, we obtain isomorphisms
\[ \begin{aligned}
M_n X & = \lim_{\Lambda[n] \to \Lambda[m]} \Map(\Lambda[m], X) \\
& = \Map({\Lambda[m] \to \Lambda[n]}{\colim} \Lambda[m], X). 
\end{aligned} \]

\begin{example}
When $n=1$, we get 
\[ \begin{aligned}
    M_1(X) & = \Map(\underset{\Lambda[0] \to \Lambda[1]}{\colim} \Lambda[0], X) \\
    & = \Lambda[0] \amalg \Lambda[0].
\end{aligned} \]
\end{example}

\begin{remark}
The previous example suggests that the colimit taken in the mapping space is some kind of ``boundary" of $\Lambda[m]$, and indeed, it is a colimit of the same representables that form $\partial \Delta[m]$ in the simplicial context.  However, in terms of geometric realization, in which we would like the boundary to consist of lower-dimensional faces, we would want the boundary to be somewhat different.  We invite the reader to explore the difference when $m=2$, for example.
\end{remark}

The previous remark notwithstanding, we use the notation 
\[ \partial \Delta[m]:= \underset{\Lambda[m] \to \Lambda[n]}{\colim} \Lambda[m], \]
where the colimit is taken over the opposite of the category $\Lambda^-{n}$, which is by duality isomorphic to $\Lambda^+(n)$.

Now, let us consider acyclic fibrations in the generalized Reedy model structure, which are those maps of cyclic spaces $X \rightarrow Y$ such that the map $X_n \rightarrow M_nX \times_{M_nY} Y_n$ is a fibration in $\SSets^{C_{n+1}}$.  For simplicity, let $P_n:= M_nX \times_{M_nY} Y_n$, and consider the pullback square defining it:
\[ \xymatrix{P_n \ar[r] \ar[d] & M_n(X) \ar[d] \\
Y_n \ar[r] & M_n(Y).} \]
Using the description above, we can write this diagram instead as
\[ \xymatrix{P_n \ar[r] \ar[d] & \Map(\partial \Lambda[n], X) \ar[d] \\
\Map(\Lambda[n], Y) \ar[r] & \Map(\partial \Lambda[n], Y). } \]

Now, the map $X_n \rightarrow P_n$ is an acyclic fibration in $\SSets^{C_{n+1}}$ if and only if it has the right lifting property with respect to the maps $\partial \Delta[m] \times C_{n+1} \rightarrow \Delta[m] \times C_{n+1}$ for any $m \geq 0$.  Here, we are regarding a $C_{n+1}$-space as a functor $Z \colon C_{n+1} \rightarrow \SSets$, where $C_{n+1}$ is treated as a category with one object.  Having acyclic fibrations defined levelwise means that the evaluation at the single object of $C_{n+1}$ gives an acyclic fibration of simplicial sets.

Thus, a map of cyclic spaces $X \rightarrow Y$ is an acyclic fibration if and only if a lift exists in any diagram of the form
\[ \xymatrix{\partial \Delta[m] \times C_{n+1} \ar[r] \ar[d] & \Map(\Lambda[n], X) \ar[d] & \\
\Delta[m] \times C_{n+1} \ar@{-->}[ur] \ar[r] & P_n \ar[r] \ar[d] & \Map(\partial \Lambda[n], X) \ar[d] \\
& \Map(\Lambda[n], Y) \ar[r] & \Map(\partial \Lambda[n],Y). } \]
Applying the adjunction between mapping spaces and products, such a lift exists if and only if a lift exists in any diagram of the form
\[ \xymatrix{\left( \partial \Delta[m] \times C_{n+1} \times \Lambda[n] \right) \cup \left( \Delta[m] \times C_{n+1} \times \partial \Lambda[n]\right) \ar[r] \ar[d] & X \ar[d] \\
\Delta[m] \times C_{n+1} \times \Lambda[n] \ar[r] \ar@{-->}[ur] & Y} \]
in the category of cyclic spaces.  The vertical maps on the left-hand side thus provide a set of generating cofibrations for the generalized Reedy structure, where we have $m,n \geq 0$.

Using an analogous argument, we can take the set of maps 
\[ \left(V[m,k] \times C_{n+1} \times \Lambda[n] \right) \cup \left(\Delta[m] \times C_{n+1} \times \Lambda[n] \right) \rightarrow \Delta[m] \times C_{n+1} \times \Lambda[n] \]
as a set of generating acyclic cofibrations, where $V[m,k]$ denotes the simplicial $k$-horn, $m \geq 0$, $0 \leq k \leq m$, and $n \geq 0$.


The sets of generating cofibrations and generating acyclic cofibrations we have described both take the form of pushout-products 
\[ \begin{tikzcd}
    A\times D\cup B\times C \arrow[r] & B\times D,
\end{tikzcd} \]
where $A\to B$ is a morphism of simplicial sets, and $C\to D$ is a morphism in $\SSets^{\Lambda}$. Standard arguments for pushout products, such as can be found in 
\cite[9.3.4]{hirsch} 
can be used to show that, for any cofibration $A\to B$ of simplicial sets and any Reedy cofibration $C\to D$ in $\SSets^{\Lambda}$, the corresponding pushout-product is a cofibration in $\SSets^{\Lambda}$ and moreover is an acyclic cofibration if either $A\to B$ or $C\to D$ is a acyclic cofibration. In particular, $\SSets^\Lambda$ has the structure of a simplicial model category. 

We summarize the preceding discussion in the following proposition.

\begin{prop}
	The generalized Reedy model structure on $\SSets^{\Lambda}$ is a cofibrantly generated simplicial model category with set of generating cofibrations 
	\[
	\{
		{\partial \Delta[m] \times C_{n+1}\times \Lambda[n]\cup \Delta[m] \times C_{n+1}\times \partial\Lambda[n]}\to {\Delta[m]\times C_{n+1}\times C_{n+1}\times \Lambda[n]}\}
	\]
	and set of generating trivial cofibrations
	\[
	\{\left(V[m,k] \times C_{n+1} \times \Lambda[n] \right) \cup \left(\Delta[m] \times C_{n+1} \times \Lambda[n] \right) \rightarrow \Delta[m] \times C_{n+1} \times \Lambda[n]\}.
	\]
\end{prop}

We conclude our discussion of the generalized Reedy model structure In analogy with the case of usual Reedy categories, the generalized Reedy model structure on $\SSets^{\Lambda}$ sits between the projective and injective model structures. 

\begin{prop}\label{prop:Reedy_fib_is_fib_on_matching}
    Let $f \colon X\to Y$ be a Reedy fibration between cyclic simplicial sets. Then $M_nX\to M_nY$ is a fibration for every $n$.
\end{prop}

Note that, since $\SSets^{C_{n+1}}$ is equipped with the projective model structure, it does not matter whether we take this matching map to be a fibration in $\SSets$ or a fibration in $\SSets^{C_{n+1}}$. 

\begin{proof}
    Following \cite[Lem. 15.3.9]{hirsch}, we define a filtration of the category $\Lambda^{-}(n)$ by declaring $F^k\Lambda^{-}(n)$ to consist of only those $\langle n\rangle \to \langle \ell\rangle$ such that $\ell \leq k$. Taking limits over this filtration gives rise to a sequence of $C_{n+1}$-spaces 
    \[ \begin{tikzcd}
    M_nX \arrow[r]\arrow[d]  &(M_nX)^{\leq n-2} \arrow[r]\arrow[d] & \cdots \arrow[r] & (M_nX)^{\leq 0}\arrow[d] \\
    M_nY \arrow[r]  &(M_nY)^{\leq n-2} \arrow[r] & \cdots \arrow[r] & (M_nY)^{\leq 0},
    \end{tikzcd} \]
    where the $C_{n+1}$-action is given by composing with the cyclic shifts of $\langle n\rangle$. 
	
    Given a ${C_{n+1}}$-cofibration $\iota \colon A\to B$ and a lifting problem 
    \[ \begin{tikzcd}
		A \arrow[r]\arrow[d,"\iota"'] & M_nX\arrow[d,"M_nf"] \\
		B\arrow[r]\arrow[ur,dashed] & M_nY,
    \end{tikzcd} \] 
    we want to construct a lift inductively. First note that the category $F^0\Lambda^{-}(n)$ is precisely the discrete category whose objects are $\langle n\rangle \to \langle 0\rangle$. Consequently, the corresponding map $(M_nX)^{\leq 0}\to (M_nY)^{\leq 0}$ is simply 
    \[ X_0^{n+1}\to Y_0^{n+1} \]
    with $C_{n+1}$-action given by permuting the factors. Since $M_0X\cong \ast\cong M_0Y$ is the terminal object, the requirement that $f$ be a Reedy fibration implies that $X_0\to Y_0$ is a fibration. Thus, $(M_nX)^{\leq 0}\to (M_nY)^{\leq 0}$ is a fibration, and so we obtain a $C_{n+1}$-map $g_0:B\to (M_nX)^{\leq 0}$ lifting the corresponding square. 
	
    Suppose now that we have defined $g_k \colon B \to (M_nX)^{\leq k}$ lifting the corresponding square. This map corresponds to a cone over the restricted diagram 
    \[ \begin{tikzcd}
    X \colon  &[-3em] F^k \Lambda^{-}(n)\arrow[r] & \SSets,
    \end{tikzcd} \] 
    and our aim is now to extend this cone to $F^{k+1}\Lambda^{-}(n)$. We denote the component of the cone at $\varphi \colon \langle n\rangle \to \langle \ell \rangle$ by $\mu_\varphi$. 
	
    Let $\psi \colon \langle n\rangle \to \langle k+1\rangle$ be an object of $F^{k+1}\Lambda^{-}(n)$ on which $\mu$ is not defined. We wish to define a corresponding component $\mu_{\psi} \colon B\to X_{k+1}$. Composition with $\psi$ defines a functor 
    \[ \begin{tikzcd}
    I_\psi \colon &[-3em]\Lambda^{-}(k+1)\arrow[r] & F^{k}\Lambda^{-}(n) 
    \end{tikzcd} \]
    and so induces a commutative diagram
    \[ \begin{tikzcd}
    (M_nX)^{\leq k}\arrow[r]\arrow[d] &  M_{k+1}X\arrow[d]\\
    (M_nY)^{\leq k}\arrow[r] &  M_{k+1}Y.
    \end{tikzcd} \]
    A morphism $\mu_\psi \colon B\to X_{k+1}$ compatible with the existing cone is equivalently a lift of the induced diagram 
    \[ \begin{tikzcd}
		A \arrow[r]\arrow[d,"\iota"'] & X_{k+1}\arrow[d]\\
		B\arrow[r]\arrow[ur,dashed,"\mu_\psi"] & M_{k+1}X\times_{M_{k+1}Y}Y_{k+1}
    \end{tikzcd} \] 
    which exists since $f$ is a Reedy fibration. 
	
    To extend the cone to all of $F^{k+1}\Lambda^{-}(n)$, we define $\mu_{\psi\circ \tau^\ell}$, where $\tau$ is the cyclic shift automorphism of $\langle n\rangle$, to be the composite of $\mu_\psi$ with the corresponding cyclic shift of $B$. We perform this procedure for each orbit of morphisms $\langle n\rangle \to \langle k+1\rangle$ under the $C_{n+1}$-action, thereby obtaining an extension of the cone to $F^{k+1}\Lambda^{-}(n)$.
	
    By induction, we thus obtain a lift of the original diagram 
    \[ \begin{tikzcd}
		A \arrow[r]\arrow[d,"\iota"'] & M_nX\arrow[d,"M_nf"] \\
		B\arrow[r]\arrow[ur,dashed,"\mu"] & M_nY
    \end{tikzcd} \]
    and so $M_nX\to M_nY$ is a fibration. 
\end{proof}

\begin{prop} \label{prop:inj_proj_Reedy}
    The identity functors  
    \[ \begin{tikzcd}
		\left(\SSets^{\Lambda}\right)^{\on{proj}} \arrow[r,"\on{id}"] &  \left(\SSets^{\Lambda}\right)^{\on{Reedy}}\arrow[r,"\on{id}"] & \left(\SSets^{\Lambda}\right)^{\on{inj}}
    \end{tikzcd} \]
    define left Quillen equivalences.
\end{prop} 

\begin{proof}
    Since the classes of weak equivalences in all three model structures agree, it suffices to show that the identity functors above preserve cofibrations, or that their inverse identity functors preserve fibrations.
	
    First, consider a generating Reedy cofibration
    \[ \begin{tikzcd}
		{\partial \Delta[m] \times C_{n+1}\times \Lambda[n]\cup \Delta[m] \times C_{n+1}\times \partial\Lambda[n]}\arrow[r] & {\Delta[m]\times C_{n+1}\times \Lambda[n]}
    \end{tikzcd} \]
    and notice that the component of this transformation at an objct $\langle k\rangle$ of $\Lambda$ is a pushout-product of cofibrations of simplicial sets. Consequently, the generating Reedy cofibrations are, in particular, injective cofibrations, and thus the right-hand adjunction is a Quillen functor. 
	
    For the left-hand adjunction, suppose that $f \colon X\to Y$ is a Reedy fibration. We can factor $f_n$ as  
    \[ \begin{tikzcd}
	X_n \arrow[r] & M_nX\times_{M_nY}Y \arrow[r] & Y_n. 
    \end{tikzcd} \]
    Since $f$ is a Reedy fibration, the first of these maps is a fibration, and by Proposition \ref{prop:Reedy_fib_is_fib_on_matching}, the second is also. Thus, the map $f_n \colon X_n\to Y_n$ is always a fibration of simplicial sets, and so the identity functor 
    \[ \begin{tikzcd}
		(\SSets^\Lambda)^{\on{Reedy}} \arrow[r,"\on{id}"]& (\SSets^{\Lambda})^{\on{proj}}
    \end{tikzcd} \]
    is right Quillen, completing the proof. 
\end{proof}

As an immediate corollary, we obtain one final property of the generalized Reedy model structure.

\begin{cor}
    The projective and generalized Reedy model structures on $\SSets^\Lambda$ are left proper. 
\end{cor}

\begin{proof}
    The model structure on $\SSets$ is left proper, and so the injective model structure is as well. Since the injective model structure is left proper, and the classes of weak equivalences in all three model structures in Proposition \ref{prop:inj_proj_Reedy} are the same, it follows from the fact that the cofibrations in the projective and generalized Reedy model structures are injective cofibration that these model structures are left proper.
\end{proof}

\section{Model structures for cyclic Segal spaces} \label{Segalmcs}

We now wish to localize the model structures in the previous section with respect to an analogue of the maps used to obtain the Segal space model structure on simplicial spaces.   We begin by reviewing the techniques of localization of model categories, then proceed to  use them to develop our desired model structure.

\subsection{Local objects and localization}

We understand a \emph{localization} of a model category at a collection $\mathcal{S}$ of morphisms to be a model category with the same underlying category, but with weak equivalences obtained by formally adding $\mathcal{S}$ to the weak equivalences of the original model structure. To have better control of this procedure, we will use the framework of \emph{left Bousfield localization} to produce our desired model structures. 

Loosely speaking, the left Bousfield localization arises by formally keeping the same cofibrations as the original model structure when we add $\mathcal{S}$ to the weak equivalences. There are many ways to formally characterize left Bousfield localizations, but for our purposes the most convenient is via local objects and local equivalences. Since the model structures we wish to consider are simplicially enriched, we will discuss the localization procedure in this setting. 

In any simplicial model structure, given weak equivalence $A\to B$ and a fibrant object $X$, the induced map on derived mapping spaces 
\[ \begin{tikzcd}
    R\!\Map(B,X)\arrow[r] & R\!\Map(A,X) 
\end{tikzcd} \]
is a weak equivalence in the model structure on simplicial sets.  We can define the left Bousfield localization by reversing this observation to obtain candidate classes for the fibrant objects and weak equivalences of our localized model structure. In what follows, let $\mathcal M$ be a simplicial  model category and $S$ the set of morphisms at which we want to localize. 

\begin{definition}
    A fibrant object $X$ of $\mathcal{M}$ is \emph{$S$-local} if, for every morphism $A\to B$ in $S$, the induced map on derived mapping spaces 
    \[ \begin{tikzcd}
		R\!\Map(B,X)\arrow[r] & R\!\Map(A,X) 
    \end{tikzcd} \]
    is a weak equivalence. 
	
    A morphism $f \colon C\to D$ in $\mathcal{M}$ is an \emph{$S$-local equivalence} if, for every $S$-local object $Y$, the induced map on derived mapping spaces 
    \[ \begin{tikzcd}
		R\!\Map(C,Y)\arrow[r] & R\!\Map(D,Y) 
    \end{tikzcd} \]
    is a weak equivalence. 
\end{definition}

The idea is that, if the localized model structure exists, then the fibrant objects in the localized model structure must be $S$-local objects. Assuming that the fibrant objects are precisely the $S$-local objects, then the weak equivalences must be the $S$-local equivalences. We thus arrive at the following definition. 

\begin{definition}
    The \emph{left Bousfield localization} of $\mathcal{M}$ at the set $S$, if it exists, is the model structure uniquely determined by the requirements that 
	\begin{itemize}
		\item the cofibrations are precisely the cofibrations in the original model structure $\mathcal M$; and 
  
		\item the weak equivalences are the $\mathcal{V}$-local equivalences. 
	\end{itemize}
\end{definition}

The key result that we need for Bousfield localizations is the following result from \cite[Theorem 4.46]{barwick}.

\begin{theorem} \label{thm:Bousfield_existence}
    Let $\mathcal{M}$ be a left proper, combinatorial, simplicial model category such that the sources of the generating cofibrations and trivial cofibrations are cofibrant. Let $\mathcal{S}$ be a set of morphisms in $\mathcal{M}$. Then the left Bousfield localization $L_S\mathcal{M}$ of $\mathcal{M}$ at $S$ exists and is a left proper, combinatorial, simplicial model category. 
\end{theorem}

\subsection{The Segal localization}

Let us first recall the localization used to obtain a model structure for Segal spaces, so that we can modify it appropriately to get a cyclic analogue.

Let $G(n)$ denote the simplicial set given by the colimit 
\[ \Delta[1] \overset{d_0}{\rightarrow} \Delta[0] \overset{d_1}{\leftarrow} \Delta[1] \overset{d_0}{\rightarrow} \cdots \overset{d_1}{\leftarrow} \Delta[1], \]
where there are $n$ copies of $\Delta[1]$.  There is a natural inclusion $G(n) \rightarrow \Delta[n]$; the image is sometimes called the \emph{spine} of $\Delta[n]$.  Regarding this inclusion as a map of discrete simplicial spaces, if we map into a fixed simplicial space $X$, then we get maps
\[ X_n = \Map(\Delta[n], X) \rightarrow \Map(G(n),x) = \underbrace{X_1 \times_{X_0} \cdots \times_{X_0} X_1}, \]
called \emph{Segal maps}.  

A model structure for Segal spaces is obtained by localizing either the injective or the projective model structure on simplicial spaces with respect to the maps $G(n) \rightarrow \Delta[n]$ for $n \geq 2$.  As a result, the fibrant objects are precisely the simplicial spaces $X$ that are fibrant in the underlying  model structure and for which the Segal maps are weak equivalences of simplicial sets for $n \geq 2$.

We want to give a similar localization to obtain a model structure on the category of cyclic spaces in which the fibrant objects are \emph{cyclic Segal spaces}, or cyclic spaces whose underlying simplicial spaces are Segal spaces.

First, let us make the appropriate definitions.  Let $\Gamma(n)$ denote the cyclic set that is the colimit of the diagram
\[ \Lambda[1] \overset{d_0}{\rightarrow} \Lambda[0] \overset{d_1}{\leftarrow} \Lambda[1] \overset{d_0}{\rightarrow} \cdots \overset{d_1}{\leftarrow} \Lambda[1], \]
which has a natural inclusion into the representable cyclic set $\Lambda[n]$.  As before, we treat this map $\Gamma[n] \rightarrow \Lambda[n]$ as a map of discrete cyclic spaces.

We can now localize any one of our model structures on the category of cyclic spaces with respect to the maps $\Gamma(n) \rightarrow \Lambda[n]$ for all $n \geq 2$.  It remains to show that this model structure has the desired fibrant objects, namely, the cyclic Segal spaces that are fibrant in the appropriate original model structure.

We denote the set of morphisms at which we localize by 
\[ \Seg_\Lambda:=\{\Gamma(n)\to \Lambda[n]\}_{n\geq 2}. \]
To identify the fibrant objects in the localized model structure, it suffices to identify the $\Seg_\Lambda$-local fibrant objects. Since the class of fibrant objects in the injective model structure contains the fibrant objects in the other two, we can thus identify the $\Seg_\Lambda$-local fibrant objects in the other two model structures. 

\begin{prop}
    An injectively fibrant object $X$ of $\SSets^{\Lambda}$ is $\Seg_\Lambda$-local if and only if it is a cyclic Segal space. 
\end{prop}

\begin{proof}
    Since the object $X$ is injectively fibrant, and all objects are injectively cofibrant, the derived mapping spaces $R\Map(\Gamma(n),X)$ and $R \Map(\Lambda[n], X)$ are equivalent to the corresponding underived mapping spaces. Consequently, $X$ is $\Seg_\Lambda$-local if and only if, for every $n\geq 2$, the maps of simplicial sets 
    \[ \begin{tikzcd}
		\Map(\Lambda[n],X) \arrow[r] & \Map(\Gamma(n),X) 
    \end{tikzcd} \] 
    are weak equivalences. Applying the representability of $\Lambda[n]$ and the fact that $\Map$ sends colimits in the first variable to limits, we obtain a commutative diagram 
    \[ \begin{tikzcd}
		\Map(\Lambda[n],X)\arrow[d,"\cong"'] \arrow[r] & \Map(\Gamma(n),X)\arrow[d,"\cong"] \\
		X_n \arrow[r] & {X_1\times_{X_0}X_1\times_{X_0}\cdots \times_{X_0} X_1 }.
    \end{tikzcd} \]
    The bottom horizontal map is precisely the Segal map for $X$, and thus, we see that $X$ is $\Seg_\Lambda$-local if and only if $X$ is a cyclic Segal space. 
\end{proof}

Applying Theorem \ref{thm:Bousfield_existence}, we obtain our desired localization.  

\begin{theorem}
    Localizing the injective, generalized Reedy, or projective model structure on $\SSets^{\Lambda}$ at the set $\Seg_\Lambda$ produces a left proper, combinatorial, simplicial model category whose fibrant objects are the cyclic Segal spaces that are fibrant in the original model structure.  
\end{theorem}

\section{Cyclic 2-Segal spaces} \label{2Segalmcs}

We can analogously study models for cyclic 2-Segal spaces, i.e. cyclic spaces whose underlying simplicial space has the 2-Segal property property defined by Dyckerhoff and Kapranov in \cite[\S 2.3]{dk}. 

Let us start with the definition of 2-Segal spaces, which were independently introduced under the name \emph{decomposition spaces} by G\'alvez-Carrillo, Kock, and Tonks in \cite[\S 1.3]{gkt}.  The 2-Segal conditions generalize the Segal conditions, and there are multiple ways of seeing this generalization. 

Recall from the previous section that the Segal maps are induced by the spine inclusion $G(n)\to \Delta[n]$. We can view $G(n)$ as a triangulation of the interval with precisely $n+1$ vertices. Moving up one dimension we can consider polygons instead of intervals, and triangulations of these same with a fixed set of vertices. We can consider an $(n+1)$-gon $P_{n+1}$ and consider triangulations of $P_{n+1}$ which have the same vertices as $P_{n+1}$. In the case of the square, we have precisely two triangulations, as pictured in Figure \ref{fig:triang_squares}.
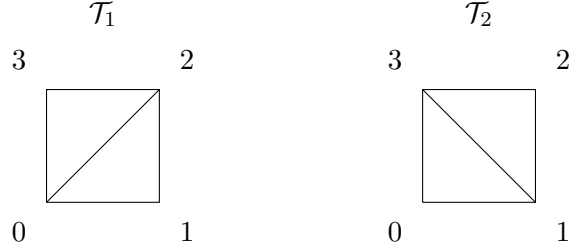
\begin{figure}
    \begin{tikzpicture}
		\draw (0,0) -- (1.5,0) -- (1.5,1.5) -- (0,1.5) -- (0,0); 
		\path (0,0) node[label=below left:$0$] {};
		\path (1.5,0) node[label=below right:$1$] {};
		\path (1.5,1.5) node[label=above right:$2$] {};
		\path (0,1.5) node[label=above left:$3$] {};
		\draw (0,0) to (1.5,1.5); 
		
		\path (0.75,2.5) node {$\mathcal{T}_1$};
		
		\begin{scope}[xshift=5cm]
			\draw (0,0) -- (1.5,0) -- (1.5,1.5) -- (0,1.5) -- (0,0); 
			\path (0,0) node[label=below left:$0$] {};
			\path (1.5,0) node[label=below right:$1$] {};
			\path (1.5,1.5) node[label=above right:$2$] {};
			\path (0,1.5) node[label=above left:$3$] {};
			\draw (1.5,0) to (0,1.5);
			\path (0.75,2.5) node {$\mathcal{T}_2$};
		\end{scope}
    \end{tikzpicture}
    \caption{The two triangulations of the square.} \label{fig:triang_squares}
\end{figure}

We can consider the triangles of these triangulations as 2-simplices in $\Delta[n]$, and thus consider $\mathcal{T}_1$ and $\mathcal{T}_2$ as simplicial subsets of $\Delta[n]$, consisting of those 2-simplices in the triangulation, together with their faces and degeneracies. As simplicial sets, the triangulations of the square are
\[ \begin{tikzcd}
    \mathcal T_1 \colon & 3 & 2 \arrow[l] & \mathcal T_2 \colon & 3 & 2 \arrow[l] \\ 
    & 0 \arrow[u] \arrow[r] \arrow[ur] & 1 \arrow[u] && 0 \arrow[u] \arrow[r] & 1. \arrow[u] \arrow[ul]
\end{tikzcd} \]
Treating these simplicial sets as discrete simplicial spaces, we can map into a fixed simplicial space $X$ to get maps 
\[  \begin{tikzcd}
    \Map_{\SSets^{\Deltaop}}(\mathcal T_1, K) &  \Map_{\SSets^{\Deltaop}}(\Delta[3], K) \arrow[l]\arrow[r] & \Map_{\SSets^{\Deltaop}}(\mathcal{T}_2, K) 
\end{tikzcd} \]
precisely as we did with $G(n)\to \Delta[n]$. These induced maps can be rewritten as 
\[ \begin{tikzcd}
    K_2 \times_{K_1} K_2 & K_3 \arrow[l]\arrow[r] & K_2 \times_{K_1} K_2.
\end{tikzcd} \]
Note that the two pullbacks here are different, since the simplicial subsets of $\Delta[3]$ used the define them are different, as indicated by the triangulations.

More generally, let us call a triangulation of an $(n+1)$-gon which has the same vertices as that $(n+1)$-gon a \emph{triangulation by vertices}. Then we can define a 2-Segal simplicial space as follows. 

\begin{definition}
    A \emph{2-Segal space} is a simplicial space $K$ such that for every $n \geq 3$ and every triangulation of a regular $(n+1)$-gon by vertices, the induced map
    \[ K_n \to \underbrace{K_2 \times_{K_1} \cdots \times_{K_1} K_2}_{n-1} \]
    is a weak equivalence of simplicial sets.
\end{definition}

Roughly speaking, the condition that $K$ be 2-Segal tells us that we can obtain the space of $n$-simplices of $K$ by gluing together 2-dimensional data according to a triangulation of a polygon. 

In parallel to this geometric characterization of the 2-Segal conditions, there is another, more algebraic interpretation. Compared with the up-to-homotopy categories provided by Segal spaces, we get a similar, but weaker structure from a 2-Segal space $K$. The simplicial space still provides spaces of ``objects'' $K_0$ and ``morphisms'' $K_1$, but if we consider the span
\[ \begin{tikzcd}
    K_1\times_{K_0}K_1 & K_2 \arrow[l,"{(d_2,d_0)}"'] \arrow[r,"{d_1}"] & K_1,
\end{tikzcd} \]
which in a Segal space would define the composition, we see that the map $(d_2,d_0)$ no longer need be a weak equivalence. However, this ``composition'' is still associative, at least up to homotopy, as can be seen from looking at the two inclusions of the triangulations of the square into $\Delta[3]$ above. 

As the preceding discussion suggests, the 2-Segal conditions do, indeed generalize the Segal conditions, as the following proposition makes precise. 

\begin{prop} \cite[Prop. 2.3.3]{dk} \label{prop:1-Seg_2-Seg}
    Every Segal space $K$ is a 2-Segal space.
\end{prop}

While we do not replicate the proof here, we can comment on the key point. Given a triangulation by vertices $\mathcal{T}$ of an $(n+1)$-gon in the plane, and assuming the vertices of the $(n+1)$-gon are labeled counterclockwise by $0, \ldots, n$, we have inclusions of simplicial sets 
\[ \begin{tikzcd}
    G(n)\arrow[r] & \Delta[\mathcal{T}] \arrow[r] & \Delta[n] 
\end{tikzcd} \]  
that, in turn, induce maps 
\[ \begin{tikzcd}
    \Map(G(n),K) & \Map(\Delta[\mathcal{T}],K)\arrow[l] & \Map(\Delta[n],K), \arrow[l]
\end{tikzcd} \]
where $\Delta[\mathcal T]$ is the simplicial set corresponding to the triangulation $\mathcal T$, using the labeling of the vertices.  The right-hand map is the one we must show is an equivalence to show that $K$ is 2-Segal, and the composite is one of the Segal maps. The proof then proceeds by using the Segal conditions to show that the left-hand map is an equivalence.

With these preliminaries in mind, we can return to cyclic spaces. 

\begin{definition}
    A cyclic space $X \colon \Lambda^{\op}\to \SSets$ is \emph{2-Segal} if its underlying simplicial space $j^\ast X$ is a 2-Segal space.
\end{definition}

In \cite[\S 5]{dk}, Dyckerhoff and Kapranov describe a model structure for 2-Segal spaces, given by localizing the injective model structure on simplicial spaces at the set of maps 
\[ \begin{tikzcd}
\Delta[\mathcal{T}] \arrow[r] & \Delta[n]
\end{tikzcd} \]
for $n\geq 3$ and $\mathcal{T}$ a triangulation of the $(n+1)$-gon by vertices. The corresponding local objects are then precisely the injectively fibrant 2-Segal spaces. 

We can follow a similar procedure to describe cyclic 2-Segal spaces.  Let $\Lambda[\mathcal T]$ be the union of copies of $\Lambda[2]$ as prescribed by a triangulation of an $(n+1)$-gon for $n \geq 3$, which has an induced map to $\Lambda[n]$.  We localize the levelwise model structures with respect to all such maps to obtain model structures for cyclic 2-Segal objects. 

More formally, let $2\Seg_\Lambda$  denote the set of all morphisms of the form
\[ \Lambda[\mathcal{T}]\to \Lambda[n] \]
for $n\geq 3$ and $\mathcal{T}$ a triangulation of the $(n+1)$-gon by vertices. Applying Theorem \ref{thm:Bousfield_existence} as we did for Segal spaces then yields the desired model structures. 

\begin{theorem}
    Localizing the injective, generalized Reedy, or projective model structure on the category of cyclic spaces with respect to the set $2\Seg_\Lambda$ yields a left-proper, combinatorial, simplicial model category whose fibrant objects are the 2-Segal objects that are fibrant in the respective original model category.
\end{theorem}

As with cyclic Segal spaces, we again can compare the resulting model structures. 

\begin{prop}
    The identity functors define left Quillen equivalences 
    \[ \begin{tikzcd} L_{2\Seg_\Lambda}\left(\SSets^{\Lambda}\right)^{\proj} \arrow[r,"\on{id}"] &  L_{2\Seg_\Lambda}\left(\SSets^{\Lambda}\right)^{\Reedy}\arrow[r,"\on{id}"] & L_{2\Seg_\Lambda}\left(\SSets^{\Lambda}\right)^{\inj}
    \end{tikzcd}\]
    between localized model structures. 
\end{prop}

Moreover, Proposition \ref{prop:1-Seg_2-Seg} allows us to compare the 1-Segal and 2-Segal model structures. 

\begin{prop}
    Fixing one of the levelwise model structures on cyclic spaces, the identity functor defines left Quillen functors on the corresponding localizations 
    \[ \begin{tikzcd}
		\SSets^{\Lambda} \arrow[r] & L_{2\Seg_\Lambda} \left(\SSets^{\Lambda} \right) \arrow[r] & L_{\Seg_\Lambda} \left(\SSets^{\Lambda} \right).
	\end{tikzcd} \]
\end{prop}


\begin{thebibliography}{9}
\bibitem[Bar10]{barwick}  C. Barwick. On left and right model categories and left and right Bousfield localizations. Homology Homotopy Appl. 12 (2) 245 - 320, 2010. 

\bibitem[BM11]{bm}
Clemens Berger and Ieke Moerdijk, On an extension of the notion of Reedy category, \emph{Math.\ Z.} 269 (2011) no.\ 3-4, 977--1004.

\bibitem[BOORS21]{boors}
Julia E.\ Bergner, Ang\'elica M.\ Osorno, Viktoriya Ozornova, Martina Rovelli, and Claudia I.\ Scheimbauer, 2-Segal objects and the Waldhausen construction, \emph{Alg.\ Geom.\ Topol.} 21 (2021) 1267--1326.

\bibitem[BS]{BSupcoming}
Julia E.\ Bergner and Walker H. Stern, Cyclic 2-Segal sets and the Waldhausen construction, \emph{in preparation}

\bibitem[Con83]{connes}
A.\ Connes, Cyclic cohomology and functors $\Ext^n$, \emph{C.R.\ Acad.\ Sci.\ Paris} 296 (1983), 953--958.

\bibitem[DHK85]{dhk}
Dwyer, Hopkins, and Kan, The homotopy theory of cyclic sets, \emph{Trans.\ Amer.\ Math.\ Soc.} 291 (1985), no.\ 1, 281–-289.

\bibitem[DK15]{dkcsg} 
T.\ Dyckerhoff and M.\ Kapranov, Crossed simplicial groups and structured surfaces, \emph{Contemp.\ Math.}, 643, 2015.

\bibitem[DK19]{dk}
T.\ Dyckerhoff and M.\ Kapranov, Higher Segal spaces, \emph{Lecture Notes in Mathematics}, 2244. Springer, 2019.

\bibitem[FL91]{fl} 
Z.\ Fiedorowicz, J.-L.\ and Loday, Crossed Simplicial groups and their associated homology, \emph{Trans.\ Amer.\ Math.\ Soc.} 326, No.\ 1, 57-87, 1991.

\bibitem[GKT18]{gkt} I. G\'alvez-Carrillo, J. Kock, and A. Tonks.  Decomposition spaces, incidence algebras and Möbius inversion I: Basic theory.Advances in Mathematics 331, 2018

\bibitem[GJ93]{gj} Ezra Getzler and John D.S. Jones, The  cyclic homology  of crossed  product  algebras, Journal für die reine und angewandte Mathematik, vol. 1993, no. 445, 1993, pp. 161-174. 

\bibitem{heller}
Alex Heller, Homotopy theories, \emph{Mem.\ Amer.\ Math.\ Soc.} 71 (1988), no.\ 383.


\bibitem[Hir03]{hirsch}
Philip S.\ Hirschhorn, \emph{Model Categories and Their Localizations, Mathematical Surveys and Monographs 99}, AMS, 2003.


\bibitem[Kra87]{kras} 
R.\ Krasauskas, Skew-simplicial groups,  \emph{Lith.\ Math.\ J.} Vol.\ 27, 47–54, 1987.

\bibitem[Lod91]{loday} 
J.-L.\ Loday, \emph{Cyclic Homology}, Grundlehren der Mathematischen Wissenschaften 301. Springer, 1991.

\bibitem{lurie}
Jacob Lurie, \emph{Higher Topos Theory. Annals of Mathematics Studies}, 170. Princeton University Press, Princeton, NJ, 2009.

\bibitem[Qui67]{quillen}
Daniel Quillen, \emph{Homotopical Algebra, Lecture Notes in Math 43}, Springer-Verlag, 1967.

\bibitem[Ree]{reedy}
C.L.\ Reedy, Homotopy theory of model categories, unpublished manuscript, available at http://www-math.mit.edu/\verb1~1psh.

\bibitem[Rez01]{rezk}
Charles Rezk, A model for the homotopy theory of homotopy theory, \emph{Trans.\ Amer.\ Math.\ Soc.} 353(3), 2001, 973--1007.

\bibitem[SN18]{scholzenikolaus} T. Nikolaus and P. Scholze, On topological cyclic homology. \emph{Acta mathematica.} 221 (2) 2018, 203--409.

\bibitem[S19]{stern} W. Stern, 2-Segal objects and algebras in spans. \emph{J. Homotopy Relat. Str.} 16 (2021) pp. 297--361.

\end{thebibliography}
\end{document}